\documentclass[reqno,11pt]{amsart}

\usepackage{mathdots}
\usepackage{tikz}
\usepackage{tikz-cd}
\usepackage{amssymb}
\usepackage{amsgen}
\usepackage{amsmath}
\usepackage{amsthm}
\usepackage{mathrsfs}
\usepackage{cite}
\usepackage{amsfonts}

\newcommand{\End}{\mathop{\mathrm{End}}}

\newcommand{\skel}[1]{^{(#1)}}

\newcommand{\dom}{\mathop{\boldsymbol d}}
\newcommand{\ran}{\mathop{\boldsymbol r}}

\newcommand{\supp}{\mathop{\mathrm{supp}}}

\newcommand{\J}{\mathrel{\mathscr J}} 
\newcommand{\R}{\mathrel{\mathscr R}} 
\newcommand{\eL}{\mathrel{\mathscr L}} 

\newcommand{\inv}{^{-1}}
\newcommand{\p}{\varphi}

\newcommand{\ov}[1]{\ensuremath{\overline {#1}}}

\newcommand{\wh}{\widehat}

\usepackage{xcolor}

\newtheorem{Thm}{Theorem}[section]
\newtheorem{Prop}[Thm]{Proposition}

\newtheorem{Lemma}[Thm]{Lemma}
{\theoremstyle{definition}
}
{\theoremstyle{remark}
\newtheorem{Rmk}[Thm]{Remark}}
\newtheorem{Cor}[Thm]{Corollary}
{\theoremstyle{remark}
}
{\theoremstyle{remark}
}

{\theoremstyle{remark}
}
{\theoremstyle{remark}
}
{\theoremstyle{remark}
}
{\theoremstyle{remark}
\newtheorem*{Claim*}{Claim}}
\newtheorem*{Thm*}{Theorem}

\numberwithin{equation}{section}

\title[Stable finiteness of ample groupoid algebras]{Stable finiteness of ample groupoid algebras, traces and applications}

\author{Benjamin Steinberg}
\address[B.~Steinberg]{%
    Department of Mathematics\\
    City College of New York\\
    Convent Avenue at 138th Street\\
    New York, New York 10031\\
    USA}
\email{bsteinberg@ccny.cuny.edu}

\thanks{The author was supported by a PSC CUNY grant, a Simons Foundation Collaboration Grant, award number 849561, the Australian Research Council Grant DP230103184 and Marsden Fund Grant MFP-VUW2411}
\date{June 3, 2025}

\keywords{Stable finiteness, ample groupoid, inverse semigroup, Leavitt path algebra, traces}
\subjclass[2020]{20M25, 16S88, 22A22}

\begin{document}

\begin{abstract}
In this paper we study stable finiteness of ample groupoid algebras with applications to inverse semigroup algebras and Leavitt path algebras, recovering old results and proving some new ones.  In addition, we develop a theory of (faithful) traces on ample groupoid algebras, mimicking the $C^\ast$-algebra theory but taking advantage of the fact that our functions are simple and so do not have integrability issues, even in the non-Hausdorff setting.  The theory of traces is closely connected with the theory of invariant means on Boolean inverse semigroups.  It turns out that for Hausdorff ample groupoids with compact unit space, having a stably finite algebra over some commutative ring implies the existence of a tracial state on its reduced $C^*$-algebra.
We include an appendix on stable finiteness of more general semigroup algebras, improving on an earlier result of Munn, which is independent of the rest of the paper.
\end{abstract}
\maketitle

\section{Introduction}
A ring $R$ (not necessarily unital) is Dedekind finite if every left invertible element of a corner in $R$ is also right invertible; it is stably finite if each matrix ring $M_n(R)$ with $n\geq 1$ is Dedekind finite.   Stable finiteness plays an important role in $C^*$-algebra theory, particularly in the classification program of Elliott. It has also long been important in ring theory including, for instance, in the foundational work of Kaplansky~\cite{Kaplanskydirect} on stable finiteness of group rings; see also~\cite{Goodearlreg} for the role of stable finiteness in the theory of von Neumann regular rings.  For example, the $K_0$-group of a stably finite unital ring (or $C^*$-algebra) is a partially ordered abelian group with strong order unit~\cite{GoodearlHandelman}.  A fundamental result says that a stably finite unital $C^*$-algebra admits a quasitrace~\cite{quasitrace}.  One can view a result of Goodearl and Handelman~\cite{GoodearlHandelman} concerning the existence of (pseudo)rank functions on stably finite von Neumann regular rings as an analogue to the existence of quasitraces on stably finite $C^*$-algebras.  Stable finiteness also plays a role in Cohn's work on rank functions for general rings~\cite{Cohnrankfunc}, although the analogy with quasitraces is less direct here.  Initially, in pure ring theory, stable finiteness was viewed as a strengthening of the IBN (invariant basis number) property; see, for example, the discussion in~\cite[Chapter~1]{Lam2}.  A unital ring $R$ has the IBN  if the rank of a free module is well defined.  For example, the Leavitt path algebra associated to a graph with a single vertex and  $n$ edges has the free module of rank $1$ isomorphic to the free module of rank $n$~\cite{LeavittBook}, and hence does not have IBN.

This paper is concerned with the stable finiteness of algebras of ample groupoids over commutative rings with unit.  These algebras include group algebras, Leavitt path algebras~\cite{LeavittBook} and inverse semigroup algebras~\cite{mygroupoidalgebra}.  I was primarily interested in the question of whether an inverse semigroup algebra $KS$ over a field $K$ is stably finite if and only if $S$ contains no infinite idempotents (i.e., does not contain a copy of the bicyclic monoid).  As far as I know, this is an open question, even over fields of characteristic $0$.

Kaplansky~\cite{Kaplanskydirect} famously showed that a group algebra $KG$ over a field of characteristic $0$ is stably finite.  His approach was analytic.  The whole problem reduces to the complex group algebra $\mathbb CG$ which embeds in its group von Neumann algebra, which has a faithful trace.  Kaplansky observed that a $C^\ast$-algebra with a faithful trace is stably finite and hence all its subrings are stably finite. Later, Passman~\cite{Passmanidem} gave a proof staying within $\mathbb CG$, but using metric properties of $\mathbb C$; see also his book~\cite{PassmanBook}.  It is an open question, known as Kaplansky's direct finiteness conjecture,  whether group algebras over fields of positive characteristic are stably finite.  Kaplansky's conjecture is known to be true for sofic groups~\cite{soficgroup} (which might possibly include all groups), and  this is often used as a selling point for the theory of sofic groups.

Easdown and Munn~\cite{EasdownMunn} and Crabb and Munn~\cite{CrabbMunn} studied the existence of faithful traces on inverse semigroups algebras but did not address the issue of stable finiteness in their papers.  Inverse semigroup algebras are isomorphic to ample groupoid algebras and there is a theory on how to build traces on $C^\ast$-algebras of such groupoids, at least when the groupoids are Hausdorff (and have a compact unit space).  Since groupoids associated to inverse semigroups do not always have Hausdorff groupoids or compact unit spaces, we cannot just work in this setting.  But the functions in the complex algebra of an ample groupoid always have finite images and their support intersects the unit space in a Borel set,  and so it turns out that there is no problem, even in the non-Hausdorff case, to integrate them against an invariant measure on the unit space to get a trace. The criterion for faithfulness of the trace does become more complicated in the non-Hausdorff case but the difficulties are not insurmountable.   The argument of Passman can then be generalized to get stable finiteness in this setting without passing to the $C^*$-algebra.    This approach allows one to prove stable finiteness for $KS$ whenever $S$ is an inverse semigroup with the property that each $\mathscr D$-class contains only finitely many idempotents and $K$ is a field of characteristic zero.  However, I then discovered that this last result is implicit in an old paper of Munn~\cite{MunnDirect} (who proves Dedekind finiteness or direct finiteness of these types of inverse semigroup algebras).  I initially missed Munn's paper because of the many differing terminologies for the notion of Dedekind finiteness.  Munn's proof reduces things to Kaplansky's theorem.  He also proves under the above hypotheses that if the maximal subgroups of $S$ are abelian then $KS$ is stably finite over any field.  His paper works with more general semigroups than just inverse semigroups. Nonetheless, I believe the results on traces, which make up the second half of this paper, are of interest in their own right.  They also turn out to be connected with the purely algebraic results via ordered $K$-theory.

Inspired by (and improving upon) Munn's idea, I came up with a sufficient condition for stable finiteness of ample groupoid algebras over arbitrary commutative rings with unit in terms of the orbit structure and the stable finiteness of group algebras of certain isotropy groups.  In particular, I was able to improve on Munn's result and show that if $S$ is an inverse semigroup such that each $\mathscr D$-class has finitely many idempotents and $K$ is any commutative ring with unit, then $KS$ is stably finite if and only if $KG$ is stably finite for each maximal subgroup $G$ of $S$.  Applying my results to the groupoids associated to Leavitt path algebras I was able to recover, and generalize, Va\v{s}'s characterization of stably finite Leavitt path algebras~\cite{Vasdirect} (in fact, like Munn, she only considered Dedekind finiteness):
\begin{Thm*}
Let $K$ be a commutative ring with unit and $E$ a graph.  Then  $L_K(E)$ is stably finite if and only if no cycle in $E$ has an exit.  
\end{Thm*}

The main result on stable finiteness, which I'll formulate here only for Hausdorff groupoids (the correct statement for non-Hausdorff groupoids is more technical), says that if $\mathscr G$ is an ample groupoid and $K$ is a commutative ring with unit such that there is a dense invariant subset $X$ of the unit space $\mathscr G\skel 0$ such that each orbit of $X$ is closed and discrete in the induced topology, and the algebra of each isotropy group from $X$ over $K$ is stably finite, then the algebra of $\mathscr G$ over $K$ is stably finite.  The typical case is when the orbits in $X$ are all finite (i.e., there is dense set of periodic orbits with nice isotropy groups).

The purely algebraic results are in turn linked with the analytic results concerning traces. If an ample groupoid with compact unit space has a stably finite algebra over some commutative ring, then its complex algebra admits a normalized trace.  If, in addition, the groupoid is Hausdorff, then the reduced $C^*$-algebra admits a tracial state.  This can be viewed as an analogue for ample groupoid algebras of the existence of quasitraces for stably finite unital $C^*$-algebras.  I don't \textit{a priori} see any reason that stable finiteness over some commutative ring (e.g., a finite field) would imply stable finiteness over $\mathbb C$ (although stable finiteness over every finite field does, cf.~Proposition~\ref{p:malcev.idea}), but nonetheless stable finiteness over some ring is reflected in the complex algebra by the existence of this trace. This result uses that $K_0$ is partially ordered with a strong order unit for stably finite rings and results of Goodearl and Handelman~\cite{GoodearlHandelman} on the existence of states on such groups.

The paper is organized as follows.  There is a preliminary section on groupoids and inverse semigroups, followed by a section on Dedekind finite and stably finite rings.  Then the sufficient condition for stable finiteness of an ample groupoid algebra is given, followed by the applications to inverse semigroup algebras and Leavitt path algebras.  As an intermediate step relativized Cohn path algebras~\cite{LeavittBook} are realized as ample groupoid algebras with a different proof than is typically found even for the case of Leavitt path algebras.

 The second half of the paper is devoted to traces on ample groupoid algebras.  A preliminary section concerns traces and norms on $\ast$-algebras and provides a fairly general setting under which a faithful trace on a $\ast$-algebra will imply stable finiteness; the key idea comes from Passman's work~\cite{Passmanidem}.  I then discuss the connection between traces on ample groupoid algebras and invariant means~\cite{lawsonetalmean} on the unit space (already explored in the Hausdorff and unital $C^*$-context in~\cite{starlingmean} in the language of Boolean inverse monoids).  I introduce the notion of a strongly faithful invariant mean (which is the same as a faithful invariant mean for Hausdorff groupoids) and show that the existence of a strongly faithful invariant mean on an ample groupoid yields a faithful trace and stable finiteness.  For Hausdorff groupoids, existence of a faithful trace provides a strongly faithful invariant mean but it is not clear if that remains true in the non-Hausdorff case.  As a consequence, an alternate proof of the results of the first half of the paper is obtained over fields of characteristic $0$ that is more analytic in flavor and subsumes Kaplansky's result, rather than relies on it as the first half does.  In this section, I also establish that, for an ample groupoid with compact unit space, if its algebra is stably finite over some commutative ring, then its complex algebra admits a normalized trace.  If, in addition, the groupoid is Hausdorff, then its reduced $C^*$-algebra admits a tracial state.

The paper ends with an appendix that considers stable finiteness of semigroup algebras over unital rings (not necessarily commutative) and improves on Munn's result~\cite{MunnDirect}.  It recovers the earlier results for inverse semigroups using purely structural semigroup theoretic methods (like Sch\"utzenberger representations) rather than groupoids.  This appendix assume familiarity with Green's relations and other technical notions from semigroup theory.

\subsection*{Acknowledgements}
I would like to thank the referee for their careful reading of this paper and a number of helpful remarks.

\section{Groupoids and inverse semigroups}
This section contains preliminaries on groupoids and inverse semigroups.

\subsection{Inverse semigroups}
The reader is referred to~\cite{Lawson} for details on all the facts that we recall here about inverse semigroups.
An \emph{inverse semigroup} is a semigroup $S$ such that, for all $s\in S$, there is a unique $s^\ast\in S$ with $ss^\ast s=s$ and $s^\ast ss^\ast=s^\ast$.   Note that $\ast$ is an involution: $(s^\ast)^\ast=s$ and $(st)^\ast =t^\ast s^\ast$.  Thus inverse semigroups are $\ast$-semigroups where we recall that a \emph{$\ast$-semigroup} is a semigroup with an involution $\ast$ satisfying these latter two properties.
 Inverse semigroups include groups and meet semilattices.  An  important example is the inverse semigroup $I(X)$ of all partial homeomorphisms of a topological space $X$.  It consists of all homeomorphisms $f\colon Y\to Z$ between open subsets  $Y,Z\subseteq X$ under the binary operation of composition of partial mappings.  The involution is just the usual inverse homeomorphism.  Every inverse semigroup $S$ can be embedded in $I(S)$ where $S$ is endowed with the discrete topology (this is called the Preston-Wagner theorem).

The idempotents of an inverse semigroup commute and form a meet semilattice under the ordering $e\leq f$ if $ef=e=fe$.  The set of idempotents is denoted $E(S)$.  This order extends to a compatible ordering on the inverse semigroup $S$, called the \emph{natural partial ordering}, by putting $s\leq t$ if there is an idempotent $e\in E(S)$ with $s=te$ or, equivalently, if $s=ft$ for some $f\in E(S)$.  One can always choose $e=s^\ast s$ and $f=ss^\ast$.  Note that if $e\in E(S)$, then so is $ses^\ast$ for all $s\in S$.  In $I(X)$, one has $f\leq g$ if and only if $f$ is a restriction of $g$, and so you should think of the ordering as an algebraic encoding of restriction.

Homomorphisms of inverse semigroups are just semigroup homo\-mor\-phis\-ms.  They automatically preserve the natural partial order and the involution.

Two idempotents $e,f$ of an inverse semigroup $S$ are \emph{$\mathscr D$-equivalent} if there is $s\in S$ with $s^\ast s=e$ and $ss^\ast =f$.  The set of elements $s\in S$ with $ss^\ast=f$ is called the \emph{$\mathscr R$-class} $R_f$ of $f$.  If $R_f$ is finite, then there can only be finitely many idempotents $\mathscr D$-equivalent to $f$.  The \emph{maximal subgroup} $G_e$ at an idempotent $e$ is the group of all elements $s\in S$ with $ss^\ast=e=s^\ast s$.  It is the group of units of $eSe$.

An important inverse semigroup for this paper is the \emph{bicyclic monoid} $B$~\cite{CP}.  As a monoid, it has presentation $\langle a,b\mid ab=1\rangle$.  Note that $a=b^\ast$, and so it can be presented as an inverse monoid (or even as a $\ast$-semigroup) by $\langle b\mid b^\ast b=1\rangle$.  Every proper quotient of $B$ is a group and so if $x,y$ are two elements of a monoid $M$ with $xy=1$ and $yx\neq 1$, then $x,y$ generate a copy of the bicyclic monoid.  Each element of $B$ can be written uniquely in the form $b^ma^n$ with $m,n\geq 0$.  The idempotents are of the form $e_n=b^na^n$ and they form a descending chain $1=e_0>e_1>\cdots$.    All elements of $B$ belong to a single $\mathscr D$-class with a trivial maximal subgroup.  See~\cite{CP} for details.

If $R$ is a unital ring and $S$ is a semigroup, then $RS$ denotes the usual semigroup algebra of $S$;  so it has $R$-basis $S$ and product \[\sum_{s\in S}r_ss\cdot \sum_{t\in S}r'_tt=\sum_{s,t\in S}r_sr'_tst.\]  Note that if $S$ has a zero element, we do not identify the zeroes of $S$ and $RS$.

\subsection{Groupoids}

A \emph{groupoid} $\mathscr G$ is a small category in which every arrow is an isomorphism.  For example, a group is the same thing as a groupoid with a single object.  We shall write $\mathscr G$ for the set of arrows of the groupoid and identify each object with the identity arrow at that object.  The set of identities is called the \emph{unit space} and is denoted $\mathscr G\skel 0$. From this viewpoint, the domain map $\dom$ has the form $\dom(\gamma)=\gamma\inv \gamma$ and the range map $\ran$ has the form $\ran(\gamma)=\gamma\gamma\inv$.  However, at times it will be convenient to just write $\gamma\colon x\to y$ to mean $\gamma$ is an arrow from $x$ to $y$ without worrying about whether we view $x,y$ as objects or identity arrows.  This viewpoint on groupoids is not common in category theory but is prevalent in analysis and is convenient when defining function algebras on groupoids.  Note that $\mathscr G\skel 0$ is a subgroupoid of $\mathscr G$.

A \emph{topological groupoid} is a groupoid $\mathscr G$ endowed with a topology for which the multiplication map and the inversion map are continuous (where the space of composable pairs is topologized as a subspace of $\mathscr G\times \mathscr G$).  We equip $\mathscr G\skel 0$ with the subspace topology and note that $\dom$ and $\ran$ are continuous.   A topological groupoid is \emph{\'etale} if $\dom$ is a local homeomorphism.  This is well known to be equivalent to $\ran$ being a local homeomorphism, and also to the multiplication map being a local homeomorphism, cf.~\cite{resendeetale}.  In an \'etale groupoid $\mathscr G\skel 0$ is open.  Good references on \'etale groupoids are~\cite{Exel,resendeetale,Paterson}.

A (local) \emph{bisection} of a topological groupoid $\mathscr G$ is an open subset $U\subseteq\mathscr G$ such that each of $\dom|_U$ and $\ran|_U$ is a homeomorphism of $U$ with its image.  For an \'etale groupoid,  $\dom$ and $\ran$ are open, and so $U$ is a bisection if and only if $\dom|_U$ and $\ran|_U$ are injective in this case.  Some authors do not require bisections to be open and so we often add the word open for emphasis.  A topological groupoid is \'etale if and only if the bisections form a basis for the topology~\cite{resendeetale}.  The set $\Gamma(\mathscr G)$ of open bisections of an \'etale groupoid $\mathscr G$ is an inverse monoid under setwise product $UV=\{\alpha\beta\mid \alpha\in U,\beta\in V\}$ with the inverse of $U$ being taken setwise, that is, $U^\ast =U\inv = \{\gamma\inv\mid \gamma\in U\}$.  Note that if $U\in \Gamma(\mathscr G)$, then $\dom(U)=U\inv U$ and $\ran(U)=UU\inv$.  In this paper all \'etale groupoids are assumed to have locally compact Hausdorff unit spaces, but $\mathscr G$ itself is not required to be Hausdorff.  It is well known that an \'etale groupoid (in this sense) is Hausdorff if and only if $\mathscr G\skel 0$ is closed.

Following the terminology of Paterson~\cite{Paterson}, an \emph{ample groupoid} is an \'etale groupoid $\mathscr G$ whose unit space is Hausdorff and has a basis of compact open sets; again we do not require $\mathscr G$ to be Hausdorff.  We denote by $\Gamma_c(\mathscr G)$ the set of compact open bisections of $\mathscr G$; it is an inverse subsemigroup of $\Gamma(\mathscr G)$.   A groupoid is ample if and only if it has a Hausdorff unit space and a basis of compact open bisections.

If $x\in \mathscr G\skel 0$, then the \emph{orbit}  of $x$ is the set of all $y\in \mathscr G\skel 0$ such that there is an arrow $\gamma\colon x\to y$.  The orbits form a partition of $\mathscr G\skel 0$.  A subset of $\mathscr G\skel 0$ is called \emph{invariant} if it is a union of orbits.  If $X\subseteq \mathscr G\skel 0$ is invariant, then $\mathscr G|_X$ is the subgroupoid consisting of all arrows beginning (and hence ending) in $X$.  If $X$ is either open or closed, then $\mathscr G|_X$ is ample whenever $\mathscr G$ is ample.

The \emph{isotropy group} at $x\in \mathscr G\skel 0$ is the group of all arrows $\gamma\colon x\to x$.  All isotropy groups at elements in the same orbit are isomorphic.

A groupoid $\mathscr G$ is called a \emph{group bundle} if all its elements are isotropy, that is $\dom(\gamma)=\ran(\gamma)$ for all $\gamma\in \mathscr G$.

If $K$ is a commutative ring with unit and $\mathscr G$ an ample groupoid,  the algebra of $\mathscr G$ over $K$, denoted $K\mathscr G$, is the $K$-span of the indicator functions $1_U$ with $U\in \Gamma_c(\mathscr G)$ with the convolution product
\[f\ast g(\gamma) = \sum_{\alpha\beta=\gamma}f(\alpha)g(\beta);\] these algebras were first studied systematically in~\cite{mygroupoidarxiv,mygroupoidalgebra}.  When $\mathscr G$ is Hausdorff, $K\mathscr G$ consists of the locally constant mappings $f\colon \mathscr G\to K$ with compact support, but in the non-Hausdoff case things are more subtle.  One has that $1_U\ast 1_V = 1_{UV}$, and so $K\mathscr G$ is a quotient of the semigroup algebra $K\Gamma_c(\mathscr G)$. In fact, it is the quotient of $K\Gamma_c(\mathscr G)$ by the ideal generated by elements of the form $1_{U\cup V}-1_U-1_V$ with $U,V\subseteq \mathscr G\skel 0$ disjoint compact open subsets (this was first proved in~\cite{mygroupoidarxiv} for the Hausdorff case and in an unpublished note by Buss and Meyers in the general case; see~\cite{simplicity} for a proof).

If $\mathscr G$ is an ample groupoid, $U$ is an invariant open subset of $\mathscr G\skel 0$ and $X=\mathscr G\skel 0\setminus U$, then $K\mathscr G|_U$ is a two-sided ideal in $K\mathscr G$ and $K\mathscr G/K\mathscr G|_U\cong K\mathscr G|_{X}$; see~\cite{gcrccr} for details.

\subsection{The universal groupoid of an inverse semigroup}
An action of an inverse semigroup $S$ on a locally compact space $X$ is a homomorphism $\beta\colon S\to I(X)$. We assume that the action is nondegenerate, meaning that the domains of the elements of $\beta(S)$ cover $X$. We write $\beta_s$ for $\beta(s)$ and write $\exists \beta_s(x)$ to mean that $\beta_s(x)$ is defined.  The \emph{groupoid of germs} $\mathscr G=S\ltimes X$  of the action is $\{(s,x)\mid \exists \beta_s(x)\}/{\sim}$ where $(s,x)\sim (t,y)$ if $x=y$ and there exists $u\leq s,t$ such that $\exists \beta_u(x)$.  We write $[s,x]$ for the class of $(s,x)$ (which is sometimes referred to as the \emph{germ} of $s$ at $x$).  Notice that if $[s,x]=[t,x]$, then $\beta_s(x)=\beta_t(x)$.  The product $[s,x][t,y]$ is defined if and only if $x=\beta_t(y)$, in which case it is $[st,y]$.   It is straightforward to verify that $\mathscr G$ is a groupoid and the identities are the elements of the form $[e,x]$ with $e\in E(S)$ and $\exists \beta_e(x)$.  Moreover, $[e,x]\mapsto x$ gives a bijection between $\mathscr G\skel 0$ and $X$.  Identifying $\mathscr G\skel 0$ with $X$, one has $\dom([s,x])=x$ and $\ran([s,x])=\beta_s(x)$.  In particular, $[s,x]$ is in the isotropy group at $x$ if and only if $\beta_s(x)=x$, i.e., $s$ fixes $x$.  A topology is given on $\mathscr G$ by taking as a basis sets of the form $(s,U) = \{[s,x]\mid x\in U\}$ where $U$ is an open subset of the domain of $\beta_s$.  With this topology, $\mathscr G$ is a topological groupoid.    The sets $(s,U)$ are bisections, and so the groupoid $\mathscr G$ is \'etale with unit space homeomorphic to $X$.  In particular, if $X$ has a basis of compact open sets, then $\mathscr G$ is ample.

If $S$ is an inverse semigroup, then its universal groupoid $\mathscr G(S)$  (due to Paterson~\cite{Paterson}) is the groupoid of germs for the action of $S$ on the spectrum of its semilattice of idempotents $\widehat{E(S)}$. Here $\widehat{E(S)}$ is the set of all nonzero homomorphisms (characters) $\theta\colon E(S)\to \{0,1\}$ where $\{0,1\}$ is a semigroup under multiplication.  The space $\widehat{E(S)}$ is topologized as a subspace of $\{0,1\}^{E(S)}$ and it has a basis of compact open sets.   If $e\in E(S)$, then $D(e) = \{\theta\in \widehat{E(S)}\mid \theta(e)=1\}$ is a compact open set, and the subsets  of the form $D(e)\setminus (D(f_1)\cup\cdots \cup D(f_n))$ with $f_1,\ldots, f_n\leq e$ form a subbasis for the topology consisting of compact open sets.  The action of $S$ on $\widehat{E(S)}$ is given by $\beta_s\colon D(s^\ast s)\to D(ss^\ast)$ with $\beta_s(\theta)(e) = \theta(s^\ast es)$.    Note that the orbits of $\mathscr G(S)$ are just the orbits of $S$ on $\widehat{E(S)}$ under the identification of $\mathscr G(S)\skel 0$ with $\widehat{E(S)}$.

The universal groupoid $\mathscr G(S)$ is Hausdorff if and only if, for each $s\in S$, there is a finite set $e_1,\ldots, e_n$ of idempotents with $e_i\leq s$, for $i=1,\ldots, n$, such that if $f\in E(S)$ with $f\leq s$, then $f\leq e_i$ for some $i$; see~\cite[Theorem~5.17]{mygroupoidalgebra}.

 The \emph{principal character} $\theta_e$ associated to an idempotent $e\in E(S)$ is defined by
\[\theta_e(f) =\begin{cases} 1, & f\geq e\\ 0, & \text{else.}\end{cases}\]  Of course, $\theta_e=\theta_f$ if and only if $e=f$.
  The orbit of a principal character associated to an idempotent $e$ under $S$ consists precisely of the principal characters of the idempotents in the $\mathscr D$-class of $e$ and the isotropy group at $\theta_e$ is the maximal subgroup $G_e$ of $S$ at $e$. In fact, the full subgroupoid with object set the principal characters is isomorphic as a groupoid~\cite[Lemma~5.16]{mygroupoidalgebra} to the underlying groupoid of the inductive groupoid of the inverse semigroup discussed in~\cite{Lawson}, from which these facts follow.

The following isomorphism is established in~\cite[Theorem~6.3]{mygroupoidalgebra}.

\begin{Thm}
If $S$ is an inverse semigroup, then $KS\cong K\mathscr G(S)$ for any commutative ring with unit $K$.
\end{Thm}

\section{Dedekind finite and stably finite rings}
In this paper, rings are not assumed to have identities.
An idempotent $e$ of a ring $R$ is \emph{finite} if every right invertible element of $eRe$ is left invertible.   Trivially, $0$ is a finite idempotent.  The following is well known.

\begin{Prop}\label{p:below.finite}
Let $e\in R$ be a finite idempotent and $f\in eRe$ an idempotent.  Then $f$ is finite.
\end{Prop}
\begin{proof}
Let $u,v\in fRf\subseteq eRe$ with $uv=f$.  Then $(u+(e-f))(v+(e-f)) = uv+(e-f)=e$.  Therefore, $u+(e-f)$ is invertible with inverse $v+(e-f)$, and hence $e=(v+(e-f))(u+(e-f)) = vu+e-f$.  Thus we have $vu=f$, as required.
\end{proof}

If all idempotents of $R$ are finite, then $R$ is said to be \emph{Dedekind finite}.  The terms ``von Neumann finite'' and ``directly finite'' are also commonly used. Note that $R$ is Dedekind finite if and only if it does not have a multiplicative subsemigroup isomorphic to the bicyclic monoid.  One says that $R$ is \emph{stably finite} if $M_n(R)$ is Dedekind finite for all $n\geq 1$. A unital ring is Dedekind finite if and only if every left invertible element is right invertible by Proposition~\ref{p:below.finite}.

Of course any subring of a Dedekind finite (respectively, stably finite) ring is Dedekind finite (respectively, stably finite).  A direct limit of Dedekind finite (stably finite) rings is well known to be the same.

\begin{Prop}\label{p:directlimit}
Let $R$ be a ring and suppose that $R=\varinjlim_{\alpha\in D} R_{\alpha}$ is a direct limit of Dedekind finite (respectively, stably finite) rings $R_{\alpha}$.  Then $R$ is Dedekind finite (respectively, stably finite).
\end{Prop}
\begin{proof}
It is enough to handle the Dedekind finite case, as  $M_n(R)$ is the direct limit of the $M_n(R_{\alpha})$. This is basically a consequence of the fact that the bicyclic monoid is a finitely presented semigroup and any homomorphism from a finitely presented semigroup into a direct limit of semigroups factors through one of the semigroups in the system.  In detail,  assume that each $R_{\alpha}$ is Dedekind finite, and let $e\in R$ be an idempotent.  Suppose that $u,v\in eRe$ with $uv=e$.  Then there is $\alpha\in D$ and elements $e',u',v'\in R_{\alpha}$ such that $e'$, $u'$ and $v'$ map to $e$, $u$, and $v$ (respectively) in $R$, and  $e'=(e')^2$, $e'u'e'=e'$, $e'v'e'=v'$ and $u'v'=e'$ (since the equations $eue=u$, $eve=v$, $e^2=e$ and $uv=e$ involve only finitely many elements of $R$).  But then  $v'u'=e'$ because $R_{\alpha}$ is Dedekind finite.  Therefore, $vu=e$.  This completes the proof.
\end{proof}

It is also standard that a right or left Noetherian unital ring is stably finite.  This follows from the fact that $M_n(R)$ is the endomorphism ring of $R^n$ and the following well-known lemma, whose proof we omit.

\begin{Lemma}
Let $M$ be a Noetherian module over a ring $R$.  Then $\mathrm{End}_R(M)$ is Dedekind finite.
\end{Lemma}

It follows that any commutative ring is stably finite.  Indeed, since adjoining an identity does not affect commutativity, we may assume that our commutative ring is unital.  But then  it is a direct limit of finitely generated commutative rings with unit, which are Noetherian, by Hilbert's basis theorem, and hence stably finite. Alternatively,  unital commutative rings are stably finite  thanks to determinants.

The crux of the following argument goes at least as far back as Kaplansky~\cite{Kaplanskydirect}.

\begin{Cor}\label{c:field.is.c}
Let $\mathscr G$ be an ample groupoid and $K$ an integral domain of characteristic $0$.  If $\mathbb C\mathscr G$ is Dedekind finite (respectively, stably finite), then $K\mathscr G$ is as well.
\end{Cor}
\begin{proof}
First of all, replacing $K$ by its field of quotients if necessary we may assume that $K$ is a field (since these properties pass to subrings).
We can write $K$ as a directed union of subfields $\{K_{\alpha}\mid \alpha\in A\}$ with each $K_{\alpha}$ a finitely generated field extension of $\mathbb Q$.  Since any $f\in K\mathscr G$ assumes only finitely many values, it follows that $f\in K_{\alpha}\mathscr G$ for some $\alpha\in A$.  Thus $K\mathscr G$ is the directed union of the $K_{\alpha}\mathscr G$.    By Proposition~\ref{p:directlimit}, it suffices to show that each $K_{\alpha}\mathscr G$ is Dedekind finite (respectively, stably finite).  So without loss of generality, we may assume that $K$ is finitely generated over $\mathbb Q$. But it is well known that any finitely generated extension of $\mathbb Q$ embeds in $\mathbb C$.  Indeed,  let $T$ be a transcendence basis for $K$ over $\mathbb Q$ and note that $T$ is countable since $K$ is countable.  Since $\mathbb C$ has uncountable transcendence degree over $\mathbb Q$, we can find an injective $\mathbb Q$-algebra homomorphism from $\mathbb Q[T]$ into $\mathbb C$ and this extends to an embedding $\mathbb Q(T)\hookrightarrow \mathbb C$.  Since $K/\mathbb Q(T)$ is algebraic and $\mathbb C$ is algebraically closed, we can extend this embedding to $K$.
\end{proof}

A very similar argument, shows that if $\mathscr G$ is an ample groupoid with a stably finite algebra over every finite field, then $\mathscr G$ has a stably finite algebra over every commutative ring.  For group algebras, this is a folklore argument and the same proof works for ample groupoids.   If $K$ is a unital commutative ring, then the nilpotent elements of $K$ form an ideal called the radical of $K$, which can alternatively be characterized as the intersection of all the prime ideals of $K$~\cite{Eisenbud}.   If the radical of $K$ vanishes, then $K$ is said to be reduced.  Every integral domain is reduced.   A commutative ring is Jacobson if each prime ideal is an intersection of maximal ideals. We remark that if $\p\colon K\to L$ is a homomorphism of commutative rings and $\mathscr G$ is an ample groupoid, then there is an induced homomorphism $K\mathscr G\to L\mathscr G$ given by $\p(f)(\gamma)=\p(f(\gamma))$ since $\p(1_U)=1_U$ for $U\in \Gamma_c(\mathscr G)$.

\begin{Prop}\label{p:malcev.idea}
Let $\mathscr G$ be an ample groupoid such that $F\mathscr G$ is stably finite for every finite field $F$.  Then $K\mathscr G$ is stably finite for every commutative ring $K$.
\end{Prop}
\begin{proof}
We can write $K$ as a directed union of subrings that are finitely generated as $\mathbb Z$-algebras.  Since direct limits of stably finite rings are stably finite, we may assume without loss of generality that $K$ is a finitely generated $\mathbb Z$-algebra.  First assume that $K$ is reduced.   Every finitely generated $\mathbb Z$-algebra is Jacobson~\cite[Theorem~4.19]{Eisenbud}.  Hence $K$ is Jacobson and reduced, and so the intersection of all its the maximal ideals is $0$.     Suppose that $A,B\in M_n(K\mathscr G)$ satisfy $AB=I$, but $BA\neq I$, that is, $I-BA\neq 0$.  Let $f\in K\mathscr G$ be a nonzero entry of $I-BA$ and $k\in K$ a nonzero value taken on by $f$. Then there is a maximal ideal $\mathfrak m$ of $K$ with $k\notin \mathfrak m$.  Putting $F=K/\mathfrak m$,  the images $A',B'$ of $A$ and $B$ in $M_n(F\mathscr G)$ satisfy $A'B'=I$ but $B'A'\neq  I$.  But $F$ is finite. Indeed~\cite[Theorem~4.19]{Eisenbud} with $R=\mathbb Z$ and $S=K$ says that $\mathfrak M=\{n\in \mathbb Z\mid n\cdot 1_K\in \mathfrak m\}$ is a maximal ideal of $\mathbb Z$, and hence of the form $p\mathbb Z$ for some prime $p$, and $F$ is a finite extension of $\mathbb Z/p\mathbb Z$, whence finite. This contradicts that $F\mathscr G$ is stably finite.  If $K$ is not reduced,  then since $K$ is Noetherian by Hilbert's basis theorem, it follows that the radical $R$ of $K$ finitely generated and hence nilpotent, say $R^k=0$.  Trivially, $R\mathscr G = \{f\in K\mathscr G\mid f(\mathscr G)\subseteq R\}$ is a nilpotent ideal as $(R\mathscr G)^k=0$.  If $A,B\in M_n(K\mathscr G)$ satisfy $AB=I$, then $I-BA\in M_n(R\mathscr G)$ by the previous case since $K/R$ is reduced. But $M_n(R\mathscr G)^k=0$ and hence is a nilpotent ideal in $M_n(K\mathscr G)$.  Thus $I-BA$ is nilpotent and so $BA=I-(I-BA)$ is invertible.  Thus $B$ is right invertible, and so $A=B^{-1}$. 
\end{proof}

Kaplansky famously proved~\cite[Page~122]{Kaplanskydirect} that if $G$ is a group and $K$ is a field of characteristic $0$, then $KG$ is stably finite.  Kaplansky's original proof was to show that a $C^\ast$-algebra with a faithful trace is stably finite and then to observe that the group von Neumann algebra, which contains $\mathbb CG$ as a subalgebra, has a faithful trace.  Passman later gave a proof using only undergraduate level analytic ideas~\cite{Passmanidem}.  We later adapt some of Passman's ideas to groupoid algebras.

It was proved in~\cite{soficnoetherianrings}, expanding on earlier work in~\cite{soficgroup}, that if $R$ is a right (or left) Noetherian ring and $G$ is a sofic group, then $RG$ is stably finite.  Note that it is an open question whether all groups are sofic.  Sofic groups include amenable groups (hence all abelian or solvable groups) and residually finite groups.  The reader is referred to~\cite{Weisssofic} for more details on sofic groups.

 Let us say that a unital ring $R$ is \emph{residually Noetherian} if, for every $0\neq r\in R$, there is an ideal $I$ such that $r\notin I$ and $R/I$ is  right  or left Noetherian. For example, any residually finite ring is residually Noetherian.   We say that $R$ is \emph{locally residually Noetherian} if each finitely generated unital subring  of $R$ is contained in a residually Noetherian subring of $R$, that is, any finite subset of $R$ is contained in a residually Noetherian unital subring. Note that left or right Noetherian implies residually Noetherian, which in turn implies locally residually Noetherian.  Obviously every unital commutative ring is locally residually Noetherian since every finitely generated commutative ring with unit is Noetherian by Hilbert's basis theorem.

The following theorem is~\cite[Theorem~1.2]{soficnoetherianrings}.

\begin{Thm}[Li and Liang]\label{t:LiLiang}
Let $R$ be a unital left Noetherian ring and $G$ be a sofic group. Then $RG$ is
stably finite.
\end{Thm}

As an application, we obtain the next result.

\begin{Thm}\label{t:soficnoeth}
If $R$ is a locally residually Noetherian ring, then $RG$ is stably finite for any sofic group $G$.  In particular, if
$K$ is a commutative ring with unit and $G$ is a sofic group, then $KG$ is stably finite.
\end{Thm}
\begin{proof}
Let $A,B\in M_n(RG)$ with $AB=I$ and $BA\neq I$.  Then there is a finitely generated subring $R_1$ of $R$ such that $A,B\in M_n(R_1G)$.  Since $R_1$ embeds in a residually Noetherian subring, we may assume without loss of generality that $R$ is residually Noetherian.   Let $P=I-BA$ and suppose that $p_{ij}$ is a nonzero entry of $P$. Then we can write $p_{ij}=\sum_{i=1}^n r_ig_i$ with $r_i\in R\setminus \{0\}$ and $g_i\in G$.   Then we can find an ideal $I$ of $R$ such that $r_1\notin I$ and $R/I$ is right or left Noetherian.  Put $R'=R/I$ and let $\pi\colon M_n(RG)\to M_n(R'G)$ be the homomorphism induces by the projection $R\to R'$.  Then, by choice of $I$, we have $0\neq\pi(P)=I-\pi(B)\pi(A)$ and $\pi(A)\pi(B)=I$.  Therefore, $M_n(R'G)$ is not Dedekind finite, contradicting Theorem~\ref{t:LiLiang}.  Thus $BA=I$, and so $RG$ is stably finite.
\end{proof}

Probably most of the results of this paper will work for groupoid algebras over arbitrary base rings, but we stick for the most part to the commutative case since most of the theory of ample groupoid algebras has only been developed in that setting.  The one exception is in the appendix where we consider more general semigroup algebras over arbitrary unital rings.

\section{Stably finite ample groupoid, inverse semigroup and Leavitt path algebras}\label{s:stable.finiteness}
This section addresses stable finitness in ample groupoids over arbitrary commutative base rings.  Later, we shall develop analytic techniques inspired by Passman~\cite{Passmanidem} and $C^\ast$-algebra theory to address the case of characteristic $0$.
The main result of this section draws its inspiration from an idea of Munn for monoid algebras~\cite{MunnDirect}, together with ideas from~\cite{gcrccr}.

\subsection{A stable finiteness result}

Let $\mathscr G$ be an ample groupoid and $K$ a commutative ring with unit.  By the \emph{support} of $f\in K\mathscr G$, we mean the set $\supp(f)=\{\gamma\in \mathscr G\mid f(\gamma)\neq 0\}$.
A subset $X\subseteq \mathscr G\skel 0$ was defined in~\cite{groupoidprimitive} to be \emph{$K$-dense} if, for each $0\neq f\in K\mathscr G$, $X\cap \ran(\supp(f))\neq \emptyset$.  It was observed in~\cite[Proposition~4.2]{groupoidprimitive} that any $K$-dense subset is dense in $\mathscr G\skel 0$, and the converse holds when $\mathscr G$ is Hausdorff.

It $\mathscr G$ is an ample groupoid, then a finite orbit in $\mathscr G$ is automatically both closed in the topology and discrete in the subspace topology since $\mathscr G\skel 0$ is Hausdorff.  We remark, that if an ample groupoid $\mathscr G$ is second countable, then any closed orbit is discrete in the subspace topology by~\cite[Lemma~5.1]{gcrccr}.

If $X$ is a set and $R$ is a ring, then $M_X(R)$ denotes the ring of $X\times X$-matrices over $R$ with only finitely many nonzero entries under usual matrix multiplication where the $xy$-entry of $AB$ is $\sum_{z\in X}a_{xz}b_{zy}$, which is finite by hypotheses.  Note that $M_X(R)\cong \varinjlim M_F(R)$ where $F$ runs over the finite subsets of $X$.  Thus if $R$ is stably finite, then $M_X(R)$ is stably finite, as well, by Proposition~\ref{p:directlimit}.  In fact, $R$ is stably finite if and only if $M_{\mathbb N}(R)$ is Dedekind finite.  Note that $M_{\mathbb  N}(R)$ is not unital even when $R$ is unital.

\begin{Thm}\label{t:finite.orbits.improved}
Let $\mathscr G$ be an ample groupoid and $K$ a commutative ring with unit.  Suppose that there is a set $\{\mathcal O_a\mid a\in A\}$ of closed orbits, each of which is discrete in the subspace topology, such that $\bigcup_{a\in A}\mathcal O_a$ is $K$-dense and $KG$ is stably finite for each isotropy group $G$ of an element in an orbit $\mathcal O_a$ with $a\in A$.  Then $K\mathscr G$ is stably finite.  In particular, if $K$ is an integral domain of characteristic $0$, or if the isotropy groups of the elements of the orbits $\mathcal O_a$ are sofic, then $K\mathscr G$ is stably finite.
\end{Thm}
\begin{proof}
Suppose that $e$ is an idempotent of $M_n(K\mathscr G)$ and  $a,b\in eM_n(K\mathscr G)e$ are such that $ab=e$.  Suppose that $ba\neq e$.  Then $f=e-ba$ is a nonzero idempotent in $eM_n(K\mathscr G)e$.  Let $f_{ij}\in K\mathscr G$ be a nonzero entry of $f$.    By assumption, there is an orbit $\mathcal O_a$ with $a\in A$ such that $\ran(\supp(f_{ij}))\cap \mathcal O_a\neq \emptyset$.  Since $\mathcal O_a$ is closed and invariant, there is a surjective ring homomorphism $\pi\colon K\mathscr G\to K\mathscr G|_{\mathcal O_a}$ given by $\pi(h) = h|_{\mathscr G|_{\mathcal O_a}}$ (see the discussion in~\cite[Pages 1596--1597]{gcrccr}).  We have an induced surjective ring homomorphism $\pi_n\colon M_n(K\mathscr G)\to M_n(K\mathscr G|_{\mathcal O_a})$ applying $\pi$ entrywise. By construction, $\pi_n(f)\neq 0$, and so $\pi_n(a),\pi_n(b)\in \pi_n(e)K\mathscr G|_{\mathcal O_a}\pi_n(e)$ with $\pi_n(a)\pi_n(b)=\pi_n(e)$ and $0\neq \pi_n(f) = \pi_n(e)-\pi_n(b)\pi_n(a)$.  Thus $K\mathscr G|_{\mathcal O_a}$ is not stably finite.

But if $G_a$ is the isotropy group of some fixed element of $\mathcal O_a$, then $K\mathscr G|_{\mathcal O_a}\cong M_{\mathcal O_a}(KG_a)$  because $\mathcal O_a$ is discrete, cf.~\cite[Proposition~5.3]{gcrccr} whose assumption that the base ring is a field is unnecessary.  Since $KG_a$ is stably finite by hypothesis, the observation before the theorem shows that  $K\mathscr G|_{\mathcal O_a}$ is stably finite.  This contradiction shows that $K\mathscr G$ is stably finite.

The final statement follows because  under these hypotheses $KG$ will be stably finite for each isotropy group $G$ of an element of an orbit $\mathcal O_a$ (for integral domains of characteristic $0$ by Kaplansky~\cite{Kaplanskydirect} and for sofic group algebras by Theorem~\ref{t:soficnoeth}).
\end{proof}

Since each orbit in a group bundle is a singleton, we obtain the following corollary.

\begin{Cor}\label{c:grp.bund}
Let $\mathscr G$ be an ample group bundle and $K$ a commutative ring with unit.  Suppose that there is a $K$-dense subset $X$ of $\mathscr G\skel 0$ such that $KG$ is stably finite for each isotropy group $G$ of an element of $X$.  Then $K\mathscr G$ is stably finite.  In particular, $K\mathscr G$ is stably finite when $K$ is an integral domain of characteristic $0$ or if each isotropy group of $\mathscr G$ is sofic.
\end{Cor}

We next apply our results to  obtain new proofs, and generalizations, of results of Munn for inverse semigroup algebras~\cite{MunnDirect} and Va\v{s} for Leavitt path algebras~\cite{Vasdirect}.

\subsection{Inverse semigroup algebras}
We first apply our results to inverse semigroup algebras to prove a generalization of a result of Munn~\cite{MunnDirect}.

 The following was proved in~\cite[Proposition~5.18]{groupoidprimitive}, where we recall that $\mathscr G(S)$ denotes the universal groupoid of $S$.

\begin{Prop}\label{p:prin.k.dense}
Let $S$ be an inverse semigroup and $K$ a commutative ring with unit.  Then the principal characters are $K$-dense in $\mathscr G(S)\skel 0$.
\end{Prop}

This can also be deduced for the case $K=\mathbb C$ by a result in~\cite{Skandalis}.

Next we aim to characterize when the orbit of a principal character is closed.

\begin{Lemma}\label{l:princ.closed}
Let $S$ be an inverse semigroup and $e\in E(S)$ an idempotent.  Then the following are equivalent.
\begin{enumerate}
\item The orbit of the principal character $\theta_e$ is closed.
\item For each idempotent $f\in E(S)$, there are only finitely many idempotents of the $\mathscr D$-class of $e$ that are below $f$.
\end{enumerate}
Moreover, under these conditions the orbit of $\theta_e$ is discrete in the subspace topology.
\end{Lemma}
\begin{proof}
We begin by recalling the well-known fact that if a $\mathscr D$-class $D$ contains two (strictly) comparable idempotents $f<f'$, then $D$ contains a countably infinite, strictly descending chain of idempotents below $f'$.  Indeed, since $f$ is $\mathscr D$-equivalent to $f'$, there is $s\in S$ with $ss^\ast=f$ and $s^\ast s=f'$. Then $s$ generates a bicyclic inverse submonoid of $f'Sf'$ and so the idempotents $f_n=s^n(s^\ast)^n$ with $n\geq 1$ form a strictly descending chain of idempotents in $f'Sf'$, all belonging to $D$.

Let $\mathcal O_e$ be the orbit of $\theta_e$ and let $D$ be the $\mathscr D$-class of $e$.  Suppose first that $\mathcal O_e$ is closed.  We claim that no two idempotents of $D$ are comparable. Indeed, if this is not the case, we have just seen that there is an infinite descending  chain of idempotents $f_1>f_2>\cdots$ in $D$.  Define $\theta\in \widehat{E(S)}$ by \[\theta(x) = \begin{cases}1, & \text{if}\ x\geq f_n, \text{for some}\ n\geq 1\\ 0, & \text{else.}\end{cases}\] We claim that $\theta  =\lim \theta_{f_n}$ and hence belongs to $\ov{\mathcal O_e}$.  Indeed, if $x\in E(S)$, then $\theta(x)=1$ if and only if $x\geq f_n$ for some $n\geq 1$, if and only if there exists $n\geq 1$ so that $x\geq f_m$ for all $m\geq n$, since the $f_i$ form a descending chain.  Thus the sequence $\theta_{f_n}$ converges pointwise to $\theta$.  Since $\mathcal O_e$ is closed, we must have $\theta=\theta_x$ with $x\in D$ an idempotent.  Then since $\theta(f_n)=1$ for all $n\geq 1$, we have $f_n\geq x$ for all $n\geq 1$.  On the other hand, since $\theta_x(x)=1$, we must have $x\geq f_N$ for some $N\geq 1$.  Thus $x=f_N$ and $f_n\geq f_N$ for all $n\geq 1$.  This contradicts that $f_n$ is a strictly descending chain. Thus the idempotents of $D$ are incomparable.

Next we claim that the induced topology on $\mathcal O_e$ is discrete.  Indeed, if $f\in D$, then $D(f)\cap \mathcal O_e=\{\theta_f\}$ since if $f'\in D$ and $\theta_{f'}\in D(f)$, then $f\geq f'$ and hence $f=f'$ by incomparability of idempotents in $D$.    This also proves the ``moreover'' statement.

Now let $f\in E(S)$.  Since $\mathcal O_e$ is closed, $D(f)\cap \mathcal O_e$ is compact.   But $\mathcal O_e$ is also discrete.  Thus $D(f)\cap \mathcal O_e$ is finite.  But if $f'\in D$, then $\theta_{f'}\in D(f)$ if and only if $f'\leq f$.  Thus $f$ is above only finitely  many idempotents of $D$.   This proves that (1) implies (2).

Next assume that each idempotent $f\in E(S)$ is above only finitely many idempotents of $D$.  By the remark in the first paragraph of the proof, we may conclude that idempotents of $D$ are incomparable. Let $\theta\in \widehat{E(S)}\setminus \mathcal O_e$.  Suppose first that $\theta(x)=0$ for all idempotents $x\in D$.  Choose $f\in E(S)$ with $\theta(f)=1$ and let $f_1,\ldots, f_n$ be the finitely many idempotents of $D$ below $f$ (possibly $n=0$).  Then $\theta\in D(f)\setminus \bigcup_{i=1}^nD(f_i)$.  Moreover, if $\theta_x$ belongs to this neighborhood, then $x\leq f$ and $x\neq f_1,\ldots, f_n$ and so $x\notin D$.  Therefore, $(D(f)\setminus \bigcup_{i=1}^n D(f_i))\cap \mathcal O_e=\emptyset$.

If $f'\in D$ is an idempotent with $\theta(f')=1$, then since $\theta\neq \theta_{f'}$, we must have that $\theta(x)=1$ for some $x$ with $x\ngeq f'$.  Then $g=xf'<f'$ and $\theta(g)=1$.  Consider $D(g)$, which is a neighborhood of $\theta$.  If $f\in E(S)$ and $\theta_f\in D(g)$, then $f\leq g<f'$.  Therefore, $f\notin D$ as the idempotents of $D$ are incomparable.   Thus $D(g)\cap \mathcal O_e=\emptyset$.  We deduce that $\mathcal O_e$ is closed.  This completes the proof.
\end{proof}

\begin{Rmk}\label{r:dincorner}
 We remark that if $S$ is an inverse semigroup and $e\in E(S)$ is an idempotent, two idempotents $f,f'\in eSe$ are $\mathscr D$-equivalent in $S$ if and only if they are $\mathscr D$-equivalent in $eSe$.  Indeed, if $ss^\ast=f$ and $s^\ast s=f'$, then $s=ss^\ast s\in fS\cap Sf'\subseteq eS\cap Se = eSe$.
 \end{Rmk}

We may now prove our first main generalization of Kaplansky's theorem.  The characteristic zero case can also be deduced from a result of~\cite{MunnDirect}; we discuss this in more detail after Corollary~\ref{c:Munn.result}.

\begin{Thm}\label{t:finitely.many.idems.D}
Let $K$ be a commutative ring with unit and $S$ an inverse semigroup  such that $eSe$ has finitely many idempotents in each of its $\mathscr D$-classes for all $e\in E(S)$.  Then $KS$ is stably finite if and only if $KG$ is stably finite for each maximal subgroup $G$ of $S$.  In particular, $KS$ is stably finite for any integral domain $K$ of characteristic zero, or if each maximal subgroup of $S$ is sofic.
\end{Thm}
\begin{proof}
If $G$ is a maximal subgroup of $S$, then $KG$ is a subring of $KS$ and hence if $KS$ is stably finite, then $KG$ is stably finite.  Let us turn to the converse and suppose that $KG$ is stably finite for each maximal subgroup $G$.

In light of Remark~\ref{r:dincorner}, the hypothesis that $eSe$ has finitely many idempotents in each of its $\mathscr D$-classes is equivalent to each $\mathscr D$-class of $S$ having only finitely many idempotents below $e$.  Therefore, our hypothesis, in conjunction with Lemma~\ref{l:princ.closed}, implies that each orbit of principal characters is closed, and is discrete in the subspace topology.  The isotropy groups of the principal characters are exactly the maximal subgroups by~\cite[Lemma~5.16]{mygroupoidalgebra}.  Since the principal characters are $K$-dense by Proposition~\ref{p:prin.k.dense}, we may apply Theorem~\ref{t:finite.orbits.improved} to deduce that $KS$ is stably finite if $KG$ is stably finite for each maximal subgroup $G$ of $S$.
\end{proof}

As a corollary we generalize and clarify a result of~\cite{MunnDirect}.

\begin{Cor}\label{c:Munn.result}
Let $S$ be an inverse semigroup whose $\mathscr D$-classes each contain finitely many idempotents and $K$ a commutative ring with unit.  Then $KS$ is stably finite if and only if $KG$ is stably finite for each maximal subgroup $G$ of $S$.  In particular, $KS$ is stably finite for any integral domain $K$ of characteristic zero, or if each maximal subgroup of $S$ is sofic.
\end{Cor}
\begin{proof}
Clearly, if each $\mathscr D$-class of $S$ has finitely many idempotents, then the same is true for $eSe$ for each idempotent $e\in E(S)$, and so Theorem~\ref{t:finitely.many.idems.D} applies.
\end{proof}

We remark that Theorem~\ref{t:finitely.many.idems.D} can, in fact, be deduced from Corollary~\ref{c:Munn.result}.  Indeed, if $S$ is an inverse semigroup such that $eSe$ has finitely many idempotents in each of its $\mathscr D$-classes for every idempotent $e\in E(S)$, then the same is true for any finitely generated inverse subsemigroup of $S$.  Then since stably finite rings are closed under direct limits and $S$ is the direct limit of its finitely generated inverse subsemigroups, we may assume without loss of generality that $S$ is finitely generated.  But we claim that if $S$ is finitely generated and $eSe$ has finitely many idempotents in each of its $\mathscr D$-class for all $e\in E(S)$, then $S$ has finitely many idempotents in each of its $\mathscr D$-classes. This argument is due to Pedro V.~Silva\footnote{private communication}.   For suppose some $\mathscr D$-class $D$ of $S$ contains infinitely many idempotents.  Let $A$ be a finite set of generators for $S$ closed under $\ast$.  Then infinitely many idempotents of $D$ that can be expressed as a product beginning with the same element $a\in A$.  But then $aa^\ast e=e$, that is, $e\leq aa^\ast$, for any of these idempotents $e$.  In light of Remark~\ref{r:dincorner}, it follows that $aa^\ast Saa^\ast$ has a $\mathscr D$-class with infinitely many idempotents, a contradiction.

Munn proved in~\cite{MunnDirect} that if $S$ is a von Neumann regular monoid (this is a larger class than inverse monoids) having finitely many idempotents in each $\mathscr D$-class, and $K$ is a field of characteristic $0$, or $K$ is any field and each maximal subgroup of $S$ is abelian,  then $KS$ is Dedekind finite.  If $S$ is an inverse semigroup and $B_n$ is the inverse semigroup of $n\times n$-matrix units, together with $0$, then $S\times B_n$ is an inverse semigroup with the same maximal subgroups as $S$ (up to isomorphism) and if $S$ has finitely many idempotents in each of its $\mathscr D$-classes, then the same is true for $S\times B_n$.  But $KB_n\cong M_n(K)\times K$ and so $K[S\times B_n]\cong KS\otimes_K KB_n\cong KS\otimes_K (M_n(K)\times K)\cong M_n(KS)\times KS$.  Thus if $S$ meets the hypothesis of Corollary~\ref{c:Munn.result}, then so does $S\times B_n$  and hence it suffices to prove Dedekind finiteness in Corollary~\ref{c:Munn.result} to obtain stable finiteness (as $M_n(KS)$ is isomorphic to a subring of $K[S\times B_n]$).  But adjoining an identity to an inverse semigroup does not change the property of having finitely many idempotents in each $\mathscr D$-class, and the only new maximal subgroup is trivial,  and so Corollary~\ref{c:Munn.result} for fields of characteristic $0$, or when the maximal subgroups are abelian  and $K$ is any field, follows from Munn's theorem and these trivial observations.  We shall improve on Munn's theorem in the appendix, which will also give an alternate, semigroup theoretic proof of Corollary~\ref{c:Munn.result} (although the fundamental idea is the same).

A Clifford semigroup is an inverse semigroup with central idempotents.  An inverse semigroup is Clifford if and only if each $\mathscr D$-class has exactly one idempotent.  Hence we have the following corollary.

\begin{Cor}\label{c:clifford}
Let $S$ be a Clifford inverse semigroup and $K$ a commutative ring with unit.  Then $KS$ is stably finite if and only if $KG$ is stably finite for each maximal subgroup $G$ of $S$.  In particular, $KS$ is stably finite for any integral domain $K$ of characteristic zero, or if each maximal subgroup of $S$ is sofic.
\end{Cor}

A free inverse semigroup (or monoid) is well known to have finitely many idempotents in each $\mathscr D$-class and trivial maximal subgroups, and  hence has a stably finite algebra over any commutative ring by Corollary~\ref{c:Munn.result}.

\begin{Cor}
If $S$ is a free inverse semigroup or monoid, then $KS$ is stably finite for any commutative ring with unit $K$.
\end{Cor}

An obvious necessary condition for $KS$ to be stably finite is for $S$ to not contain a subsemigroup isomorphic to the bicyclic monoid, that is, for each idempotent $e\in E(S)$, every left invertible element of $eSe$ is also right invertible.  Such inverse semigroups are sometimes called completely semisimple in the literature.  I do not know an example of a completely semisimple inverse semigroup whose semigroup algebra is not stably finite.

\subsection{Leavitt path algebras}
Let $E=(E\skel 0,E\skel 1)$ be a (directed) graph (loops and multiple edges allowed) with vertex set $E\skel 0$ and edge set $E\skel 1$. We write $\dom(e)$ for the initial vertex of an edge (or path) and $\ran(e)$ for the terminal vertex.  Following the tradition of graph theory we compose paths from left to right, so $pp'$ is valid if $\ran(p)=\dom(p')$. In what follows we shall denote the empty path at a vertex $v$ by $v$, abusing notation.

 Our goal is to prove that if $K$ is a commutative ring with unit, then the Leavitt path algebra $L_K(E)$ is stably finite if and only if the graph $E$ is a no-exit graph. 
A graph $E$ is called \emph{no-exit} if no cycle has an exit (i.e., each vertex on a cycle has out-degree $1$).  A cycle here means a nonempty closed path in which the terminal vertex is the first repeated vertex.  This is essentially a result of Va\v{s}~\cite{Vasdirect} and was proven without using groupoids. She works with traces over $\ast$-rings that are injective on idempotents.
 Like Va\v{s}, we work in the setting of relatived Cohn path algebras in order to take direct limits. Our first task is then to realize relatived Cohn path algebras in the sense of~\cite[Section~1.5]{LeavittBook} as ample groupoid algebras.

 Leavitt path algebras are usually defined by generators and relations, but we will define them as quotients of semigroup algebras of graph inverse semigroups (these semigroup algebras are often called Cohn path algebras by people in the Leavitt path algebra community).
Let $P_E$ be the graph inverse semigroup associated to $E$~\cite{Paterson,JonesLawson}.  If $E^\ast$ denotes the set of all finite paths in $E$ (including empty paths), then $P_E$ can be identified with the inverse semigroup of all partial bijections of $E^\ast$ of the form $qx\mapsto px$ (denoted $pq^\ast$) with $p,q$ paths ending at the same vertex together with the empty mapping.   If $q$ is an empty path, we usually write $p$ and if $p$ is an empty path, we normally write $q^\ast$.  If $p,q$ are both empty paths at $v$, we just write $v$ for the corresponding element.   Note that $pq^\ast$ is really the composition of $p$ and $q^\ast$, and $q^\ast$ is the inverse partial bijection to $q$.
One can present $P_E$ as a $\ast$-semigroup with zero with generating set $E\skel 0\cup E\skel 1$ and relations saying that: $E\skel 0$ is an orthogonal set of idempotents; if $e$ is an edge, then $\dom(e)e=e=e\ran(e)$ and $e^\ast e=\ran(e)$; and $f^\ast e=0$ if $f\neq e$ for $e,f\in E\skel 1$.

 Finite paths in $E$ are in bijection with nonzero idempotents of $P_E$ via $p\mapsto pp^\ast$.  Moreover, $pp^\ast\leq qq^\ast$ if and only if $q$ is a prefix (initial segment) of $p$.  Note that $\mathscr G(P_E)$ is Hausdorff since $P_E$ is \emph{$0$-$E$-unitary} (cf.~\cite{JonesLawson}), meaning that  $0$ is the only idempotent below a non-idempotent. Indeed, if $rr^\ast\leq pq^\ast$, then $rr^\ast = rr^\ast pq^\ast$.  If $r$ is a prefix of $p$, we get $rr^\ast = pq^\ast$, and so $p=r$ and $q=r$, whereas if $p$ is a prefix of $r$, say $r=ps$, then $rr^\ast pq^\ast = rs^\ast q^\ast$ and so $r=qs$, whence $p=q$.  In either case, $pq^\ast$ is an idempotent.  It follows that $\mathscr G(P_E)$ is Hausdorff by~\cite[Theorem~5.17]{mygroupoidalgebra}.

 We write $|p|$ for the length of a finite path $p$ in $E$.

The characters of $E(P_E)$ are well known.  There is the principal character $\theta_0$ associated to $0$, which is identically $1$.  Also, for each finite path $p$ in $E$ there is the principal character $\theta_{pp^\ast}$, which we shall abbreviate to $\theta_p$, and for each right infinite path $p$ in $E$, there is a character $\theta_p$ with $\theta_p(qq^\ast)=1$ if and only if $q$ is a prefix  of $p$.  This follows easily from the fact that if neither $q$ nor $r$ are prefixes of each other, then $qq^\ast rr^\ast=0$, and so for any character $\theta\neq \theta_0$, we have either $\theta$ is principal, or, for each $n\geq 0$, there is a unique path $p_n$ of length $n$ with $\theta(p_np_n^\ast)=1$ and $p_{n-1}$ is a prefix of $p_n$.  It follows that $\theta=\theta_p$ where $p$ is the unique infinite path whose prefix of length $n$ is $p_n$ for all $n\geq 0$.  One can check that if $pq^\ast\in P_E$, then $\beta_{pq^\ast}$ is defined on $\theta_0$ and those $\theta_r$ with $q$ a prefix of $r$.  Moreover,  has $\beta_{pq^\ast}(\theta_0) =\theta_0$ and $\beta_{pq^\ast}(\theta_{qz})=\theta_{pz}$ for a finite or infinite path $z$ with initial vertex $\ran(p)$. Note that $\beta_0$ is only defined on $\theta_0$ and fixes it.

A vertex $v$ of $E$ is \emph{regular} if it is not a sink and not an infinite emitter.  If $X$ is a set of regular vertices, then the \emph{Cohn path algebra of $E$ relative to $X$} is $C_K^X(E)=KP_E/I_X$ where $I_X$ is generated by the zero of $P_E$ and the idempotents $v-\sum_{\dom(e)=v}ee^\ast$ with $v\in X$.  If $X$ consists of all the regular vertices, then the Cohn path algebra of $E$ relative to $X$ is the \emph{Leavitt path algebra} $L_K(E)$.  See the book~\cite{LeavittBook} for more on Leavitt path algebras.

Notice that the set of principal characters $\theta_p$ of finite  paths $p$  ending at a vertex of $X$ form an invariant subset since two idempotent $pp^\ast$ and $qq^\ast$ are $\mathscr D$-equivalent if and only if $p,q$ end at the same vertex.  Moreover, the principal character associated to any finite path  $p$ ending at a regular vertex $v$ is isolated.  Indeed, if  $e_1,\ldots, e_n$ are the finitely many edges leaving $v$, then $D(pp^*)\setminus \bigcup_{i=1}^n D(pe_i(pe_i)^*)$ is an open set containing only $\theta_p$.  Also $\theta_0$ is an isolated point as $D(0) = \{\theta_0\}$ and is a singleton orbit since $0$ is not $\mathscr D$-equivalent to any other idempotent.   Thus if \[Y = \widehat{E(P_E)}\setminus (\{\theta_p\mid \ran(p)\in X\}\cup \{\theta_0\}),\] then $Y$ is a closed invariant subset of $\widehat{E(P_E)}$, and so $\mathscr G_{E,X}=\mathscr G(P_E)|_Y$  is an ample groupoid.

\begin{Lemma}\label{l:find.path}
Let $v$ be a vertex.  Then $D(v)\cap Y\neq \emptyset$.
\end{Lemma}
\begin{proof}
If there is path $p$ (possibly empty) from $v$ to a vertex $w\notin X$, then $\theta_p(v)=1$ and $\theta_p\in Y$.  Else, every vertex reachable from $v$ (including itself) belongs to $X$.  Since each vertex of $X$ is not a sink, it follows that we can build an infinite path $p$ starting at $v$.  Then $\theta_p\in Y$ and $\theta_p(v)=1$.
\end{proof}

 The following proposition generalizes the well-known realization of Leavitt path algebras as ample groupoid algebras, cf.~\cite{groupoidapproachleavitt}.  The approach here seems different than what is in the literature for Leavitt path algebras in full generality in that we avoid graded uniqueness theorems.

\begin{Prop}\label{p:iso:cohn}
Let $K$ be a commutative ring with unit, $E$ a graph and $X$ a set of regular vertices of $E$.  Then
$C_K^X(E)\cong K\mathscr G_{E,X}$.
\end{Prop}
\begin{proof}
It follows from~\cite[Proposition~5.2]{groupoidprimitive} that $K\mathscr G_{E,X}$ is isomorphic to the quotient of $KP_E$ by the ideal $I$ generated by all products $\prod_{i=1}^n(f-f_i)$ with $f_1,\ldots, f_n\leq f\in E(P_E)$ and $D(f)\setminus (D(f_1)\cup\cdots \cup D(f_n))\cap Y=\emptyset$.  We show that this ideal coincides with $I_X$.

 Since $D(0)\cap Y=\emptyset$, we have that the zero of $P_E$ belongs to $I$.  If $v\in X$ and $e_1,\ldots, e_n$ are all the edges with initial vertex $v$, then $D(v)\setminus \bigcup_{i=1}^n D(e_ie_i^\ast)=\{\theta_v\}$, and $\theta_v\notin Y$ since $v\in X$.  Thus $\prod_{i=1}^n(v-e_ie_i^\ast)\in I$.  But noting that $e_ie_i^\ast e_je_j^\ast =0$ in $P_E$ for $i\neq j$, we deduce that this product is $v-\sum_{i=1}^ne_ie_i^\ast +a$ where $a$ is a scalar multiple of the zero of $P_E$. Since the zero of $P_E$ belongs to $I$, we deduce that $v-\sum_{i=1}^ne_ie_i^\ast\in I$.   It follows that $I_X\subseteq I$.  This argument also shows that $\prod_{i=1}^n(v-e_ie_i^\ast)\in I_X$ since $I_X$ also contains the zero of $P_E$.

Denote also by $\beta\colon P_E\to I(Y)$ the homomorphism obtained by restricting the action of $P_E$ on $\widehat{E(P_E)}$ to the invariant subspace $Y$.

First we claim that if $D(v)\setminus \bigcup_{i=1}^nD(p_ip_i^\ast)\cap Y=\emptyset$ with $p_1,\ldots,p_n$ non\-emp\-ty, then $\prod_{i=1}^n(v-p_ip_i^\ast)\in I_X$.  Note that if $p_i$ is a prefix  of $p_j$, then $(v-p_ip_i^\ast)(v-p_jp_j^\ast) = v-p_ip_i^\ast$.  Also, $D(p_jp_j^\ast)\subseteq D(p_ip_i^\ast)$,  and so we may remove $p_jp_j^\ast$ without changing anything.  Thus, without loss of generality, we may assume that the elements of $p_1,\ldots, p_n$ are prefix incomparable. After reordering, we may assume that $p_n$ has maximum length amongst the $p_i$.

The Hasse diagram of the set of finite paths starting at $v$ with respect to the prefix ordering is a rooted tree (with root $v$) and since each path has only finitely many prefixes, two paths are prefix comparable  if and only if there is a directed path in the Hasse diagram from one to the other.  Let $T$ be the finite subtree whose vertices are the (not necessarily proper) prefixes of $p_1,\ldots, p_n$. Notice that each of $p_1,\ldots, p_n$ are leaves of $T$ by prefix incomparability.   We prove that $\prod_{i=1}^n(v-p_ip_i^\ast)\in I_X$ by induction on the number of internal (non-leaf) vertices of $T$.

Since $D(v)\cap Y\neq \emptyset$ by Lemma~\ref{l:find.path}, we must have $n\geq 1$, and hence $v$ is an internal vertex.    First we claim that if $p$ is an internal vertex of $T$, then the endpoint of $p$ belongs to $X$.  For otherwise, $\theta_p\in Y\cap D(v)\setminus \bigcup_{i=1}^nD(p_ip_i^\ast)$.

Write $p_n=pe_0$ with $e_0$ an edge.  Since the $p_i$ are prefix incomparable, $p\neq p_i$ for any $i=1,\ldots, n$. Let $w=\ran(p)$ and note that $w\in X$ by what we have already observed.  Let $e$ be any edge with initial vertex $w$ (there is at least one, and finitely many, since $w$ is regular).  We claim that $pe=p_i$ for some $i$.  Indeed, let $w'$ be the endpoint of $e$.  Then by Lemma~\ref{l:find.path} we can find a character $\theta\in Y$ with $\theta(w')=1$.  Since $w'= (pe)^\ast (pe)$ and $Y$ is invariant, we deduce that if $\theta'=\beta_{pe}(\theta)$, then $\theta'\in Y$ and $\theta'((pe)(pe)^\ast)=1$.  Therefore, $\theta'\in D(v)$ and hence we must have $\theta'\in D(p_ip_i^\ast)$ for some $i$.  Thus $pe$ and $p_i$ are prefix comparable (they are both prefixes of some finite or infinite path).  We cannot have $p_i$ a proper prefix of $pe$ or we would have that $p_i$ is a proper prefix of $p_n$.  Thus $pe$ is a prefix of some $p_i$. Since $p_n$ was chosen to have maximum length, $|pe|=|p_n|\geq |p_i|$ and so we deduce that $p_i=pe$.  Note that if $p$ is a prefix of $p_j$, then $p_j=pe$  for some edge $e$ since $p_j\neq p$ (as $p_j$ is not a proper prefix of $p_n$), and so $|p|< |p_j|\leq |p_n|=|p|+1$.   If $p=v$, then what we have just shown is that $p_1,\ldots, p_n$ are precisely the edges emitted by $v$ and hence we have already seen that  $\prod_{i=1}^n (v-p_ip_i^\ast)\in I_X$ in the second paragraph of the proof. This is also exactly the case when there is one internal node and so it is the base case of the induction.  So we may now assume that $p\neq v$.

Let us reorder the set $p_1,\ldots, p_n$ so that $p$ is not a prefix of $p_1,\ldots, p_r$ (with $r\geq 0$) and is a prefix of $p_{r+1},\ldots, p_n$ and note that
\begin{equation}\label{eq:already.saw}
pp^\ast-\sum_{i=r+1}^np_ip_i^\ast= p(w-\sum_{\dom(e)=w}ee^\ast)p^\ast\in I_X
\end{equation}
 by what we have just shown, and since $w\in X$.  We claim that if \[Z=D(v)\setminus (D(p_1p_1^\ast)\cup\cdots\cup D(p_rp_r^\ast)\cup D(pp^\ast)),\] then $Z\cap Y=\emptyset$.     Indeed, if $\theta\in Y$ with $\theta(v)=1$, then there is some $i=1,\ldots, n$ with $\theta(p_ip_i^\ast)=1$.  If $1\leq i\leq r$, then $\theta\notin Z$.  If $\theta(p_ip_i^\ast)=1$ with $r+1\leq i\leq n$, then $p$ is a prefix of $p_i$ and so $\theta(pp^\ast)=1$.  Thus $\theta\notin Z$.  Note that the paths $p_1,\ldots, p_r,p$ are prefix incomparable by construction.  Since the tree $T'$ associated to $p_1,\ldots, p_r,p$ is obtained from $T$ by removing $p_{r+1},\ldots, p_n$, which are precisely the children of $p$, it has fewer internal nodes.  Since $p\neq v$,  $T'$ has at least one internal node.   Thus, by induction, we have $\prod_{i-1}^r(v-p_ip_i^\ast)(v-pp^\ast)\in I_X$. Since the zero of $P_E$ belongs to $I_X$ and the paths $p_1,\ldots, p_r,p$ are prefix incomparable, $I_X=\prod_{i-1}^r(v-p_ip_i^\ast)(v-pp^\ast)+I_X = v-\sum_{i=1}^rp_ip_i^\ast-pp^\ast+I_X$. But $pp^\ast+I_X = \sum_{i=r+1}^np_ip_i^\ast+I_X$ by \eqref{eq:already.saw}, and hence $v-\sum_{i=1}^np_ip_i^\ast\in I_X$.  But again, using that $p_1,\ldots, p_n$ are prefix incomparable and that $I_X$ contains the zero of $P_E$, we deduce that $\prod_{i=1}^n(v-p_ip_i^\ast) +I_X=v-\sum_{i=1}^n p_ip_i^\ast+I_X=I_X$.

To complete the proof that $I\subseteq I_X$, suppose now that \[Z=D(pp^\ast)\setminus \bigcup_{i=1}^nD(pp_i(pp_i)^\ast)\cap Y=\emptyset\] with $v$ the initial vertex of $p$ and $p_1,\ldots, p_n$ nonempty.  Then since $Y$ is invariant, $\emptyset =\beta_{p^\ast}(Z) = D(v)\setminus \bigcup_{i=1}^n D(p_ip_i^\ast)\cap Y$, and so by the previous case $a=\prod_{i=1}^n(v-p_ip_i^\ast)\in I_X$.  But then $\prod_{i=1}^n(pp^\ast -pp_i(pp_i)^\ast) = pap^\ast\in I_X$.  This completes the proof that $I=I_X$ and hence $C_K^X(E)=KP_E/I_X\cong K\mathscr G_{E,X}$.
\end{proof}

Note that the isomorphism in~\cite[Proposition~5.2]{groupoidprimitive}, in our setting, sends the coset $v+I_X$, with $v\in E_0$, to $1_{D(v)\cap Y}\in K\mathscr G_{E,X}$.  Since $D(v)\cap Y\neq \emptyset$ by Lemma~\ref{l:find.path}, we deduce that $v+I_X\neq I_X$.  Hence, if $e$ is an edge, then $e+I_X\neq I_X$, since if that were the case and if $v=\ran(e)$, then $v+I_X = e^\ast e+I_X =I_X$, which contradicts our previous observation.

We need one last lemma, which is essentially well known, except that people usually work with the path groupoid (cf.~\cite{groupoidapproachleavitt}) of $E$, which is topologically isomorphic to $\mathscr G(P_E)$.

\begin{Lemma}\label{l:isotropy}
Let $\theta\in \widehat{E(P_E)}$.  Then the isotropy group of $\theta$ is trivial unless $\theta=\theta_r$ where $r$ is an infinite path of the form $bccc\cdots$ with $b$ a finite path and $c$ a nonempty closed path, in which case the isotropy group is infinite cyclic.
\end{Lemma}
\begin{proof}
Let $G$ be the isotropy group of $\theta$.  First note that $[s,\theta_0]=[0,\theta_0]$ for all $s\in P_E$ and hence $G$ is trivial if $\theta=\theta_0$.  Next assume that $\theta =\theta_r$ with $r$ a finite path.  Then $[pq^\ast,\theta]$ makes sense if and only if $q$ is a prefix of $r$.  Moreover, if $r=qs$, then $\beta_{pq^\ast}(\theta_r) = \theta_{ps}$, and so we must have $p=q$ if $[pq^\ast,\theta]\in G$.  But then $[pp^\ast,\theta]$ is the identity at $\theta$, and so $G$ is trivial.

Next assume that $\theta=\theta_r$ where $r$ is an infinite path.  Again $[pq^\ast,\theta]$ makes sense if and only if $q$ is a prefix of $r$, and if $r=qs$, then  $\beta_{pq^\ast}(\theta_r) = \theta_{ps}$.  Thus $[pq^\ast,\theta]\in G$ if and only if $ps=qs$.  If $|p|=|q|$, this implies $p=q$ and so $[pp^\ast,\theta]$ is the identity of $G$.  So if $G$ is nontrivial, then taking inverses if necessary, we may assume that $|p|>|q|$ and so $p=qc$ for some finite nonempty path $c$, which is necessarily closed.  We then have that $qcs=ps=qs$, and so $cs=s$.  It follows that $s=ccc\cdots$, and so $r=qs=qccc\cdots$.
  It remains to show that $G$ is infinite cyclic in this case.

  Define $\rho\colon G\to \mathbb Z$ by $\rho([ab^\ast,\theta_r])=|a|-|b|$.  This is well defined because if $cd^\ast\leq ab^\ast$, then there is a finite path $x$ with $c=ax$ and $d=bx$, and so $|c|-|d|=|a|+|x|-(|b|+|x|) = |a|-|b|$.  It is a homomorphism because if $[ab^\ast,\theta_r],[cd^\ast,\theta_r]\in G$, then $[ab^\ast cd^\ast,\theta_r]\in G$.  Hence $ab^\ast cd^\ast\neq 0$, and so $b$ is a prefix of $c$ or $c$ is a prefix of $b$.  If $c=bx$, then $ab^\ast cd^\ast = axd^\ast$, and we have $|ax|-|d| = |a|-|b|+|c|-|d|$.  If $b=cy$, then $ab^\ast cd^\ast = ay^\ast d^\ast$, and $|a|-|dy| = |a|- (|d|+|b|-|c|) = |a|-|b|+|c|-|d|$.  Therefore, $\rho$ is a homomorphism.  Also, $\rho$ is injective because if $\rho[ab^\ast,\theta_r]=0$, then $|a|=|b|$.  But $a,b$ are prefixes of $r$ (as we saw above) and so $a=b$.  Therefore, $[ab^\ast,\theta_r]=[aa^\ast,\theta_r]$ is the identity of $G$.  Since every nontrivial subgroup of $\mathbb Z$ is infinite cyclic, it remains to display a nontrivial element of $G$. We claim that $[qcq^\ast,\theta_r]\in G$ is nontrivial.  Note that $\beta_{qcq^\ast}(\theta_r)=\theta_r$ and so $[qcq^\ast,\theta_r]\in G$.  Also, $\rho([qcq^\ast,\theta_r]) = |qc|-|q|=|c|>0$.  This completes the proof.
\end{proof}

  The following result is due to Va\v{s}~\cite{Vasdirect}; in fact she just proves Dedekind finiteness but the class of Leavitt path algebras she considers is easily verified to be closed under taking matrix algebras.  Stable finiteness is explicitly characterized in the special case where the graph is finite and $K$ is a field in~\cite[Theorem~4.2]{AbramsNamPhuc}. In what follows we omit writing cosets and one should interpret paths as living in the appropriate Leavitt path algebra or relativized Cohn algebra.

\begin{Thm}\label{t:vas}
Let $K$ be a commutative ring with unit and $E$ a graph.  Then the following are equivalent.
\begin{enumerate}
\item $L_K(E)$ is stably finite.
\item $L_K(E)$ is Dedekind finite.
\item $E$ is a no-exit graph.
\end{enumerate}
\end{Thm}
\begin{proof}
Obviously (1) implies (2). Assume that (2) holds and that there is a cycle with an exiting edge $e$ with initial vertex $v$.  Let $p$ be the path that goes once around the cycle starting from $v$.  Then $p,p^\ast\in vL_K(E)v$ and $p^\ast p=v$. But $pp^\ast e=0$, whereas $ve=e\neq 0$.  Therefore, $pp^\ast\neq v$ and so  $L_K(E)$ is not Dedekind finite. Thus (2) implies (3).

For (3) implies (1), we use~\cite[Theorem~1.6.10]{LeavittBook}.  A subgraph $F$ of $E$ is called \emph{complete} if whenever a vertex $v$ of $F$ emits at least one edge in $F$ and finitely many in $E$, we have that all edges of $E$ with initial vertex $v$ belong to $F$.  The set of regular vertices of a graph $F$ will be denoted $\mathrm{Reg}(F)$.  Then Theorem~\cite[Theorem~1.6.10]{LeavittBook} proves that $L_K(E)\cong \varinjlim_F C_K^{\mathrm{Reg}(F)\cap \mathrm{Reg}(E)}(F)$ where $F$ runs over the finite complete subgraphs of $E$.  In fact, in~\cite{LeavittBook} they always work over fields, but the proof of~\cite[Theorem~1.6.10]{LeavittBook} works over any base commutative ring with unit.  Indeed, it is fairly easy to see that $KP_E$ is the directed union of the $KP_F$ where $F$ runs over the finite complete subgraphs and that the generating set of $I_{\mathrm{Reg}(E)}$  is the union of the generating sets of the $I_{\mathrm{Reg}(F)\cap \mathrm{Reg}(E)}$ where $F$ runs over the finite complete subgraphs of $E$, yielding the direct limit result.
However~\cite[Theorem~1.6.10]{LeavittBook} proves the stronger result that the directed system consists of injective maps and so $L_K(E)$ is really a directed union of these relative Cohn algebras.

Next we observe that if $F$ is a finite complete subgraph of $E$, then $F$ is no-exit.  Indeed, every cycle in $F$ is a cycle in $E$ and hence has no exit in $E$, let alone $F$.  Also note that since each vertex of a cycle in $E$ has out-degree $1$, and hence is regular, it follows that each vertex of a cycle in $F$ belongs to $\mathrm{Reg}(F)\cap \mathrm{Reg}(E)$.  Since stably finite rings are closed under direct limits by Proposition~\ref{p:directlimit}, we are reduced to the case of proving stable finiteness for $C_K^X(F)$ where $F$ is a finite no-exit graph and every vertex of every cycle of $F$ belongs to $X$ (this reduction also appears in~\cite{Vasdirect}).

We claim that $\mathscr G_{F,X}\skel 0$ is finite in this case.  Assuming the claim, it will follow that each orbit is closed and discrete.  Moreover, each isotropy group $G$ is either trivial or infinite cyclic by Lemma~\ref{l:isotropy}, whence $KG$ is a commutative ring with unit and thus stably finite.  Theorem~\ref{t:finite.orbits.improved} then yields that $K\mathscr G_{F,X}\cong C_K^X(F)$ is stably finite.

Let $n$ be the number of vertices of $F$.  Any path of length $n$ must visit some vertex twice and hence enter a cycle from which it cannot exit.  Since every vertex of a cycle belongs to $X$, we deduce that if $p$ is a finite path with $\theta_p\in \mathscr G_{F,X}\skel 0$, then $p$ has length less than $n$.  Since $F$ is a finite graph, there are only finitely many paths of length less than $n$.  On the other hand, since any infinite path must enter a cycle (from which it cannot escape) within the first $n$ edges, the infinite paths are of the form $qppp\cdots$ with $|q|<n$  and $p$ the label of a cycle.  Clearly, there are only finitely  many such infinite paths since $F$ is finite and any cycle has length at most $n$ (since any prefix of length $n$ must repeat a vertex).  Thus $\mathscr G_{F,X}\skel 0$ is finite.  This completes the proof.
\end{proof}

\section{Normed $\ast$-algebras and spanning semigroups}
In this section, we are interested in stable finiteness of ample groupoid algebras over fields of characteristic $0$, and hence we may restrict our attention to the field of complex numbers by Corollary~\ref{c:field.is.c}.
 Our goal is to use traces to prove that certain $\ast$-algebras are stably finite.  Although, we shall not obtain any new stable finiteness results using traces that cannot be deduced from the results of Section~\ref{s:stable.finiteness}, the approach we use here has the advantage that  it does not rely explicitly on Kaplansky's stable finiteness result for complex group algebras (although Passman's proof idea~\cite{Passmanidem} is embedded in our approach).  Also there has been quite a bit of research about faithful traces on inverse semigroups algerbas~\cite{CrabbMunn,EasdownMunn,starlingmean} and our approach recovers a number of these results.  We are hopeful that the theory of traces on complex groupoid algebras that we initiate here will find other uses outside of the context of stable finiteness.

A complex algebra $A$ is a $\ast$-algebra if it has an involution $\ast$ such that:
\begin{itemize}
\item [(SA1)]$(a+b)^\ast =a^\ast +b^\ast$ for all $a,b\in A$;
\item [(SA2)] $(ab)^\ast=b^\ast a^\ast$ for all $a,b\in A$;
\item [(SA3)] $(a^\ast)^\ast=a$ for all $a\in A$;
\item [(SA4)] $(ca)^\ast=\ov ca^\ast$ for all $c\in \mathbb C$ and $a\in A$.
\end{itemize}

 We say that the $\ast$-algebra $A$ is a \emph{seminormed}  algebra if it is equipped with a seminorm $|\cdot |\colon A\to \mathbb R$ such that:
\begin{itemize}
  \item [(N1)] $|a+b|\leq |a|+|b|$ for all $a,b\in A$;
  \item [(N2)]  $|a|\geq 0$ for all $a\in A$;
  \item [(N3)]  $|ab|\leq |a|\cdot |b|$ for all $a,b\in A$;
  \item [(N4)] $|ca|=|c||a|$ for $c\in \mathbb C$, $a\in A$;
  \item [(N5)] $|a|=|a^\ast|$ for all $a\in A$.
\end{itemize}
We do not require that $A$ be complete with respect to the seminorm.  The $\ast$-algebra $A$ is \emph{normed} if additionally $|a|=0$ if and only if $a=0$.

If $A$ is any normed $\ast$-algebra and $I$ is a $\ast$-ideal (closed under $\ast$), then $A/I$ is a seminormed $\ast$-algebra with seminorm given by
\[|a+I| = \inf\{|b| \mid b\in a+I\}.\]  This is a norm precisely when $I$ is closed.

Recall that a $\ast$-semigroup is a semigroup $S$ with an involution $s\mapsto s^\ast$ (satisfying (SA2) and (SA3)).  For example, inverse semigroups are $\ast$-semigroups with respect to the usual involution.
A \emph{projection} in a $\ast$-semigroup is an idempotent $p$ with $p^\ast=p$. Of course, the multiplicative semigroup of a $\ast$-algebra is a $\ast$-semigroup. On the other hand, if $S$ is a $\ast$-semigroup, then $\mathbb CS$ is a normed $\ast$-algebra where if $a=\sum_{s\in S}a_ss$, then $a^\ast = \sum_{s\in S}\ov {a_s}s^\ast$ and the norm is the $\ell_1$-norm $|a|=\sum_{s\in S}|a_s|$.

If $A$ is a $\ast$-algebra, then a \emph{spanning semigroup} for $A$ is a $\ast$-subsemigroup $S\subseteq A$ that spans $A$ as a vector space.  For example, $S$ is a spanning semigroup of $\mathbb CS$.  If $S$ is a spanning semigroup for $A$, then $\pi\colon \mathbb CS\to A$ induced by the inclusion of $S$ is a surjective homomorphism of $\ast$-algebras and so $A\cong \mathbb CS/I$ where $I=\ker \pi$ is a $\ast$-ideal.  Therefore, the $\ell_1$-norm on $\mathbb CS$ induces a seminorm on $A$ turning it into a seminormed $\ast$-algebra.  Concretely,
\[|a| = \inf\left\{\sum_{s\in S}|a_s|\mathrel{\big\vert} a=\sum_{s\in S}a_ss\right\}.\]  We call this the \emph{seminorm induced} by $S$ on $A$.

  If $\mathscr G$ is an ample groupoid, there is an involution $\ast$ on $\mathbb C\mathscr G$ given by \[f^\ast(\gamma) = \ov {f(\gamma\inv)}\] for $f\in \mathbb C\mathscr G$ and $\gamma\in \mathscr G$.  Note that $1_U^*= 1_{U\inv}$.  Thus the indicator functions form a spanning semigroup  $\{1_U\mid U\in \Gamma_c(\mathscr G)\}\cong \Gamma_c(\mathscr G)$, which we tacitly identify with $\Gamma_c(\mathscr G)$ to avoid introducing new notation.  Hence we have an induced seminorm on $\mathbb C\mathscr G$.  Every function in $\mathbb C\mathscr G$ takes on only finitely many values and is, hence, bounded.  Put $|f|_{\infty} = \sup_{\gamma\in \mathscr G}|f(\gamma)|$.  Then $|\cdot |_{\infty}$ is a vector space norm on $\mathbb C\mathscr G$.

\begin{Prop}\label{p:is.normed}
Let $\mathscr G$ be an ample groupoid and let $|\cdot |$ be the seminorm on $\mathbb C\mathscr G$ induced by $\Gamma_c(\mathscr G)$.  Then $|f|_{\infty}\leq |f|$ and hence $|\cdot |$ is a norm.
\end{Prop}
\begin{proof}
If $f=\sum_{U\in \Gamma_c(\mathscr G)} c_U1_U$, then \[|f(\gamma)| = \left\vert\sum_{U\in \Gamma_c(\mathscr G)} c_U1_U(\gamma)\right\vert\leq \sum_{U\in \Gamma_c(\mathscr G)} |c_U|\] since $0\leq 1_U(\gamma)\leq 1$.  Therefore, $|f(\gamma)|\leq |f|$ and so $|f|_{\infty}\leq |f|$.  In particular, if $|f|=0$, then $|f|_{\infty}=0$, and so $f=0$.
\end{proof}

A set $E$ of idempotents in a ring $R$ is a set of \emph{local units} if each finite subset $F$ of $R$ is contained in $eRe$ for some $e\in E$.  If $R$ has a set of local units, we say that the ring $R$ has \emph{local units}.  This is equivalent to being a direct limit of unital rings with respect to ring homomorphisms that do not have to respect the identities.

If $A$ is a $\ast$-algebra and $S$ is a spanning semigroup, we say that $S$ contains a set of \emph{local projections} for $A$ if, for every finite subset $F$ of $A$, we have that $F\subseteq pAp$ for some projection $p\in S$.  For example, if $\mathscr G$ is an ample groupoid, then $\Gamma_c(\mathscr G)$ contains a set of local projections for $\mathbb C\mathscr G$. Indeed, it is well known~\cite{groupoidbundles} that for each finite subset $F\subseteq \mathbb C\mathscr G$ there is a compact open subset $U\subseteq \mathscr G\skel 0$ with $F\subseteq 1_U\mathbb C\mathscr G1_U$; moreover, the indicator function $1_U$ is a projection.  Note that if $S$ is an inverse semigroup, then it is not necessarily the case that $S$ contains a set of local projections for $\mathbb CS$, but there is always an inverse $\ast$-subsemigroup of $\mathbb CS$ containing $S$ that does.

If $A$ is unital, then $S$ contains a set of local projections for $A$ if and only if $S$ contains the identity of $A$.  Indeed, if $S$ contains the identity, it trivially contains a set of local projections.  Conversely, if $S$ contains a set of local projections, then $1\in pAp$ for some projection $p\in S$, and so $1=p1p=p\in S$.

It will be convenient to build a spanning semigroup for $M_n(A)$ from a spanning semigroup for $A$.  First note that if $A$ is a $\ast$-algebra, then $M_n(A)$ is also a $\ast$-algebra in the standard way where $(b_{ij})^\ast = (b_{ji}^\ast)$.  Let $I_n$ be the symmetric inverse monoid on the set $[n]=\{1,\ldots,n\}$.  It consists of all partial bijections $\sigma\colon [n]\to [n]$ and the involution is given by taking the inverse partial bijection.  The multiplication is composition of partial functions.   We write $\dom(\sigma)$ for the domain of $\sigma\in I_n$ and $\ran(\sigma)$ for the range.

Denote by $E_{ij}$ the standard $ij$-matrix unit.
If $S$ is a spanning semigroup for $A$, the \emph{rook semigroup} $R_n(S)$ consists of all elements of $M_n(A)$ of the form $\sum_{j\in \dom(\sigma)} s_jE_{\sigma(j)j}$ for some $\sigma\in I_n$ and $s_j\in S$.  The term ``rook matrix'' is due to Solomon~\cite{Solomonrook} and refers to the fact that each row and column of a rook matrix has at most one nonzero entry and hence the matrix can be viewed as a placement of non-attacking rooks on an $n\times n$ chessboard.

\begin{Prop}\label{p:rook.matrix}
Let $S$ be a spanning semigroup for $A$ (containing a set of local projections).  Then $R_n(S)$ is a spanning semigroup for $M_n(A)$ (containing a set of local projections).
\end{Prop}
\begin{proof}
If $B=\sum_{j\in \dom(\sigma)} s_jE_{\sigma(j)j}\in R_n(S)$, then $B^\ast = \sum_{j\in \dom(\sigma)} s_j^\ast E_{j\sigma(j)} = \sum_{i\in \dom(\sigma\inv)}s_{\sigma\inv(i)}^*E_{\sigma\inv (i)i}\in R_n(S)$.  Also, if $C=\sum_{j\in \dom(\tau)}t_jE_{\tau(j)j}$, then $BC = \sum_{j\in \dom(\sigma\tau)} s_{\tau(j)}t_jE_{\sigma\tau(j)j}\in R_n(S)$.  Since $sE_{ij}\in R_n(S)$ for all $s\in S$ and $1\leq i,j\leq n$, and $S$ spans $A$, we deduce that $R_n(S)$ spans $M_n(A)$.

Suppose that $S$ contains a set of local projections for $A$ and let $F$ be a finite set of matrices from $M_n(A)$.  Then there is a projection $p\in S$ such that all the finitely many entries of the elements of $F$ belong to $pAp$.  Then $P=\sum_{j=1}^n pE_{jj}\in R_n(S)$ (take $\sigma=1_{[n]}$) and $P$ is a projection with $F\subseteq PM_n(A)P$.
This completes the proof.
\end{proof}

\section{Faithful traces and stable finiteness}

A \emph{trace} on a $\ast$-algebra $A$ is a functional $\tau\colon A\to \mathbb C$ satisfying:
\begin{itemize}
  \item [(T1)] $\tau(ab)=\tau(ba)$ for all $a,b\in A$;
  \item [(T2)] $\tau(a^\ast)=\ov{\tau(a)}$ for all $a\in A$;
  \item [(T3)] $\tau(aa^\ast)\geq 0$ for all $a\in A$.
\end{itemize}
The trace $\tau$ is said to be \emph{faithful} if it also satisfies
\begin{itemize}
\item [(T4)] $\tau(aa^\ast)=0$ implies $a=0$.
\end{itemize}

In the case $\tau$ is a faithful trace on $A$, we can define a complex inner product $(\cdot,\cdot)\colon A\times A\to \mathbb C$ by $(a,b) = \tau(ab^*)$.  The corresponding norm on $A$ will be denoted  $\Vert\cdot \Vert$. Note that $\Vert a^\ast\Vert^2 = \tau(a^\ast a) =\tau(aa^\ast) = \Vert a\Vert^2$, and so $\Vert a^\ast \Vert=\Vert a\Vert$.    The next proposition shows that $\ast$ gives the adjoint map with respect to this inner product for both left and right translations.

\begin{Prop}\label{p:adjoint.op}
Let $\tau$ be a faithful trace on $A$ with corresponding inner product $(\cdot,\cdot)$.  Then for all $a,b,c\in A$, we have:
\begin{enumerate}
  \item $(ab,c) = (b,a^\ast c)$;
  \item $(a,bc) = (ac^\ast,b)$.
\end{enumerate}
\end{Prop}
\begin{proof}
This is standard: $(ab,c)=\tau(abc^\ast) = \tau(bc^\ast a) = \tau(b(a^\ast c)^\ast) = (b,a^\ast c)$ and $(a,bc) = \tau(a(bc)^\ast) = \tau(ac^\ast b^\ast) = (ac^\ast, b)$.
\end{proof}

Notice that if $\tau$ is a trace, then $\tau(p)=\tau(pp^\ast)\geq 0$ for each projection $p$.  If $\tau$ is faithful, then $\tau(p)>0$ for each non-zero projection.

A trace $\tau$ on a unital $\ast$-algebra is said to be \emph{normalized} if $\tau(1)=1$.

Suppose $S$ is a spanning semigroup for $A$.  We say that a trace $\tau$ on $A$ is \emph{$S$-contractive} if $\tau(as(as)^\ast)\leq \tau(aa^\ast)$ for all $a\in A$ and $s\in S$. This condition is left-right dual.

\begin{Prop}\label{p:S.contractive}
Let $\tau$ be a trace on a $\ast$-algebra $A$ with spanning semigroup $S$.  Then the following are equivalent.
\begin{enumerate}
  \item $\tau$ is $S$-contractive.
  \item $\tau(sa(sa)^\ast)\leq \tau(aa^\ast)$ for all $a\in A$, $s\in S$.
\end{enumerate}
\end{Prop}
\begin{proof}
If $\tau$ is $S$-contractive and $a\in A$, $s\in S$, then since $s^\ast \in S$, we have that $\tau(sa(sa)^\ast) = \tau(saa^\ast s^\ast) = \tau(a^\ast s^\ast sa) = \tau((a^\ast s^\ast)(a^\ast s^\ast)^\ast)\leq \tau(a^\ast a)=\tau(aa^\ast)$ because $\tau$ is $S$-contractive.  The converse  implication is similar.
\end{proof}

In the case that $\tau$ is a faithful trace, the condition that $\tau((as)(as)^\ast)\leq \tau(aa^\ast)$ can be rewritten as $\Vert as\Vert\leq \Vert a\Vert$ and the second condition in Proposition~\ref{p:S.contractive} can be rewritten as $\Vert sa\Vert\leq \Vert a\Vert$.  In other words, left and right translations by elements of $S$ are contractive with respect to $\Vert\cdot \Vert$, whence the name.  The following proposition generalizes an observation of Passman~\cite{Passmanidem} for the special case of the standard trace on a group algebra.

\begin{Prop}\label{p:Passman.inequality}
Let $A$ be a $\ast$-algebra with spanning semigroup $S$.  Suppose that $\tau$ is an $S$-contractive faithful trace on $A$ with associated norm $\Vert\cdot\Vert$. Denote by $|\cdot |$ the seminorm induced on $A$ by $S$.
\begin{enumerate}
\item $\Vert ab\Vert\leq \Vert a\Vert \cdot |b|$ for all $a,b\in A$.
\item $\Vert ab\Vert\leq |a|\cdot \Vert b\Vert$ for all $a,b\in A$.
\end{enumerate}
In particular, if $|a|=0$, then $aA=0=Aa$.  Consequently, if $A$ has local units, then $|\cdot |$ is a norm.
\end{Prop}
\begin{proof}
We handle the first item only as the second follows dually using Proposition~\ref{p:S.contractive}(2).
Since $\tau$ is $S$-contractive, we have that $\Vert as\Vert\leq \Vert a\Vert$ for all $a\in A$ and $s\in S$.  If $a=0$, there is nothing to prove, so assume that $a\neq 0$.  If we write $b=\sum_{s\in S}b_ss$, then the triangle inequality yields
\[\Vert ab\Vert =\Vert \sum_{s\in S}b_sas\Vert \leq \sum_{s\in S}|b_s|\Vert as\Vert\leq \sum_{s\in S}|b_s|\Vert a\Vert = \Vert a\Vert \sum_{s\in S}|b_s|.\]
We conclude that $\frac{\Vert ab\Vert}{\Vert a\Vert}\leq \sum_{s\in S}|b_s|$, and so taking infima over all representations of $b$ as a linear combination of elements of $S$ we obtain $\Vert ab\Vert\leq \Vert a\Vert\cdot  |b|$.

It follows immediately from (1) and (2) that if $|a|=0$, then $\Vert ab\Vert=0=\Vert ba\Vert $ for all $b\in B$ and hence $aA=0=Aa$.  If $A$ has local units, then $aA\neq 0$ for any $a\neq 0$ and so $|a|>0$ for any $a\neq 0$.
\end{proof}

We show that $\tau$ is even better behaved when $A$ is unital and $\tau$ is normalized.  In particular, the trace is a bounded linear map.

\begin{Lemma}\label{l:normalized}
Let $A$ be a unital $\ast$-algebra with spanning semigroup $S$.  Suppose that $\tau$ is a normalized $S$-contractive faithful trace on $A$ with associated norm $\Vert\cdot\Vert$. Denote by $|\cdot |$ the norm induced on $A$ by $S$.
\begin{enumerate}
  \item $\Vert 1\Vert =1$.
  \item $\Vert a\Vert\leq |a|$ for all $a\in A$.
  \item $|\tau(a)|\leq \Vert a\Vert\leq |a|$ for all $a\in A$.
\end{enumerate}
\end{Lemma}
\begin{proof}
Note that $|\cdot |$ is a norm by Proposition~\ref{p:Passman.inequality}.    For the first item, we have that $\Vert 1\Vert^2 = \tau(1\cdot 1^\ast)=\tau(1)=1$ since $\tau$ is normalized, whence $\Vert 1\Vert =1$.
 Therefore, $\Vert a\Vert = \Vert 1\cdot a\Vert\leq \Vert 1\Vert\cdot |a|=|a|$ by Proposition~\ref{p:Passman.inequality}(1), yielding the second item.  The third item follows from the first and second since by the Cauchy-Schwarz inequality, $|\tau(a)| = |\tau(a\cdot 1^\ast)|= |(a,1)|\leq \Vert a\Vert\cdot \Vert 1\Vert = \Vert a\Vert\leq |a|$.
\end{proof}

Next we construct from a faithful $S$-contractive trace on $A$, a faithful $R_n(S)$-contractive trace on $M_n(A)$.

\begin{Prop}\label{p:amplify.trace}
Let $A$ be a $\ast$-algebra and $\tau$ a (faithful) trace on $A$.  Then $\tau_n\colon M_n(A)\to \mathbb C$ defined by $\tau_n(B) = \sum_{i=1}^n\tau(b_{ii})$ is a (faithful) trace.  Moreover, if $S$ is a spanning semigroup for $A$ and $\tau$ is $S$-contractive, then $\tau_n$ is $R_n(S)$-contractive.
\end{Prop}
\begin{proof}
The fact that $\tau_n$ is a (faithful) trace is well known and standard.  The key points are $\tau_n(B^\ast) = \sum_{i=1}^n \tau(b_{ii}^\ast) = \sum_{i=1}^n\ov{\tau(b_{ii})}=\ov {\tau_n(B)}$,   $\tau_n(BC) = \sum_{i,j}\tau(b_{ij}c_{ji}) = \sum_{i,j}\tau(c_{ji}b_{ij})=\tau_n(CB)$ and
$\tau_n(BB^\ast) = \sum_{i,j} \tau(b_{ij}b_{ij}^\ast)$.

Suppose that $C=\sum_{j\in \dom(\sigma)}s_jE_{\sigma(j)j}\in R_n(S)$ and $B\in M_n(A)$.  If $BC=(d_{ij})$, then
\[d_{ij} = \begin{cases}b_{i\sigma(j)}s_j, & \text{if}\ j\in \dom(\sigma)\\ 0, & \text{else.}\end{cases}\]  Thus, using that $\tau$ is $S$-contractive, we have that
\begin{align*}
\tau_n(BC(BC)^\ast) &= \sum_{i=1}^n\sum_{j\in \dom(\sigma)}\tau(b_{i\sigma(j)}s_j(b_{i\sigma(j)}s_j)^\ast)\\ &\leq \sum_{i=1}^n\sum_{j\in \dom(\sigma)}\tau(b_{i\sigma(j)}b_{i\sigma(j)}^\ast)\leq \sum_{i,j}\tau(b_{ij}b_{ij}^\ast) = \tau_n(BB^\ast)
\end{align*}
since $\sigma$ is a partial bijection, and hence each index appears at most once as $\sigma(j)$.
\end{proof}

The proof of the next theorem is an adaptation of Passman's argument~\cite{Passmanidem} for the case of the standard trace on a group ring.  Our notion of $S$-contractive faithful trace is just an axiomatization of the properties he uses in his proof, and so many of the steps of the argument are taken almost verbatim.

\begin{Thm}\label{t:stably.finite}
Let $A$ be a $\ast$-algebra with spanning semigroup $S$ containing a set of local projections for $A$.  If there is an $S$-contractive faithful trace $\tau$ on $A$, then:
\begin{enumerate}
  \item $\tau(e)>0$ for all nonzero idempotents $e\in A$.
  \item $A$ is stably finite.
\end{enumerate}
\end{Thm}
\begin{proof}
Let us assume for the moment that (1) is true and we prove  (2).  It suffices to show that $A$ is Dedekind finite under these hypotheses since   $M_n(A)$ has spanning semigroup $R_n(S)$ containing a set of local projections for $M_n(A)$ by Proposition~\ref{p:rook.matrix} and an $R_n(S)$-contractive faithful trace  by Proposition~\ref{p:amplify.trace}, and hence satisfies these hypotheses.

Suppose that $e$ is a nonzero idempotent and  $u,v\in eAe$ with $uv=e$.  Note that $f=vu\in eAe$ is an idempotent as $f^2=vuvu=veu=vu=f$.  Also $\tau(e)=\tau(uv)=\tau(vu)=\tau(f)$ and hence $\tau(e-f)=\tau(e)-\tau(f)=0$. Thus $e-f=0$ by (1) since $e-f$ is an idempotent.  We conclude that $vu=e$ and so $u$ is invertible. Thus $e$ is finite.

The proof of (1) is more difficult.  Let $e\in A$ be a nonzero idempotent.  Since $S$ contains a set of local projections for $A$, we may find a nonzero projection $p\in S$ with $e\in pAp$.  In particular, $\tau(p)=\tau(pp^\ast)>0$.   Then $pAp$ is a unital $\ast$-algebra with identity $p$ and $pSp$ is a spanning semigroup for $pAp$.  Define $\tau'\colon pAp\to \mathbb C$ by $\tau'(a) = \tau(a)/\tau(p)$.  Then $\tau'$ is a faithful normalized trace.  It is $pAp$-contractive since if $a\in pAp$ and  $s\in pSp$, then $\tau'(as(as)^\ast)=\tau(as(as)^\ast)/\tau(p)\leq \tau(aa^\ast)/\tau(p) = \tau'(aa^\ast)$.  Thus to prove (1), we may assume without loss of generality that $A$ is unital, $S$ contains the identity of $A$ and $\tau$ is normalized.  In particular, Lemma~\ref{l:normalized} applies.    The remainder of our proof now follows~\cite{Passmanidem} very closely.  We use $\Vert\cdot\Vert$ for the norm on $A$ coming from the inner product $(a,b) = \tau(ab^\ast)$ and we use $|\cdot |$ for the norm on $A$ induced by $S$.  Notice that since $\tau$ is faithful and $A$ is unital, $|\cdot |$ is indeed a norm by Proposition~\ref{p:Passman.inequality}.   Also, note that $|1|=1$ since $1\in S$ implies $|1|\leq 1$, but $1=\Vert 1\Vert\leq |1|$ by Lemma~\ref{l:normalized}.

If $X\subseteq A$ and $a\in A$, put \[d(a,X) = \inf\{\Vert a-b\Vert\mid b\in X\}.\]

\begin{Lemma}\label{l.passman.linear}
Let $L\subseteq A$ be a linear subspace of $A$ and $a,b\in L$.  Then \[|(b,a-c)|^2\leq \Vert b\Vert ^2\left(\Vert a-c\Vert^2- d(c,L)^2\right)\] for all $c\in A$.
\end{Lemma}
\begin{proof}
If $b=0$, then there is nothing to prove so assume that $b\neq 0$ and put $k=(a-c,b)/\Vert b\Vert^2$.  Then $a-kb\in L$, and so $\Vert a-kb-c\Vert\geq d(c,L)$.  Therefore,
\begin{align*}
\Vert a-c\Vert^2-d(c,L)^2& \geq \Vert a-c\Vert^2- \Vert a-kb-c\Vert^2\\ & = \Vert a-c\Vert^2 - (a-c-kb,a-c-kb)\\
&= \Vert a-c\Vert^2 - \Vert a-c\Vert^2+\ov k(a-c,b) +k(b,a-c)-k\ov k\Vert b\Vert^2\\
&= k\ov k\Vert b\Vert^2+k\ov{(a-c,b)}-k\ov k\Vert b\Vert^2\\ & = k\ov k\Vert b\Vert^2 = \frac{|(a-c,b)|^2}{\Vert b\Vert^2}.
\end{align*}
The lemma now follows.
\end{proof}

Put $R=eA$; note that $R$ is a linear subspace and a right ideal.   Set $d=d(1,R)$.  Then we can choose, for each $n>0$, an element $r_n\in R$ such that $\Vert 1-r_n\Vert^2<d^2+1/n^4$.    Note that $er_n=r_n$ for all $n\geq 1$.

\begin{Lemma}\label{l:passman.limits}
There exists positive constants $M,M'$  such that:
\begin{enumerate}
  \item $|\Vert r_n\Vert^2-\tau(r_n)|< M/n$;
  \item $\Vert r_ne-e\Vert< M'/n$.
\end{enumerate}
\end{Lemma}
\begin{proof}
First note that $\Vert r_n-1\Vert< \sqrt{d^2+1/n^4}\leq d+1$, an inequality that we shall use throughout. Now we prove (1).  Observe that \[\Vert r_n\Vert\leq \Vert r_n-1\Vert+\Vert 1\Vert < d+1+1=d+2\] by Lemma~\ref{l:normalized}.  By Lemma~\ref{l.passman.linear} with $L=R$, $a=b=r_n$ and $c=1$, we have that \[|(r_n, r_n-1)|^2\leq \Vert r_n\Vert^2(\Vert r_n-1\Vert^2-d(1,R)^2) < (d+2)^2(1/n^4),\] and so $|(r_n,r_n-1)|< (d+2)/n^2\leq (d+2)/n$.  But \[(r_n,r_n-1) = (r_n,r_n)-(r_n,1) = \Vert r_n\Vert^2-\tau(r_n),\] and so (1) holds with $M=d+2$.

For the second item, we compute  that
\[\Vert r_ne-e\Vert^2 = ((r_n-1)e,(r_n-1)e) = ((r_n-1)ee^\ast,r_n-1)\] using Proposition~\ref{p:adjoint.op}.  Since $r_n\in R$ and  $(r_n-1)ee^\ast = (r_ne-e)e^\ast\in R$, we have by Lemma~\ref{l.passman.linear} with $a=r_n$, $b=(r_n-1)ee^\ast$ and $c=1$ that
\[|((r_n-1)ee^\ast,r_n-1)|^2\leq \Vert (r_n-1)ee^*\Vert^2(\Vert r_n-1\Vert^2-d(1,R)^2)<\Vert (r_n-1)ee^\ast\Vert^2/n^4.\] We conclude that
\[\Vert r_ne-e\Vert^2< \Vert (r_n-1)ee^\ast\Vert/n^2\leq \Vert r_n-1\Vert\cdot  |ee^\ast|/n^2\leq (d+1)|ee^\ast|/n^2\] by Proposition~\ref{p:Passman.inequality}.  Thus (2) holds with  $M' = \sqrt{(d+1)|ee^\ast|}$.
\end{proof}

A final lemma will finish the proof of Theorem~\ref{t:stably.finite}(1).

\begin{Lemma}\label{l:passman.trace.formula}
Retaining the above notation:
\begin{enumerate}
\item $\tau(e) = \lim_{n\to \infty} \Vert r_n\Vert^2$ and hence is a nonnegative real number;
\item $\tau(e)\geq \Vert e\Vert^2/|e|^2>0$.
\end{enumerate}
\end{Lemma}
\begin{proof}
We compute $|\tau(e)-\Vert r_n\Vert^2|\leq |\tau(e)-\tau(r_n)|+|\tau(r_n)-\Vert r_n\Vert^2|$.  By Lemma~\ref{l:passman.limits}(1), there is a positive constant $M$ with $|\tau(r_n)-\Vert r_n\Vert^2|< M/n$.  Since $\tau(r_n)=\tau(er_n)=\tau(r_ne)$, we have that $|\tau(e)-\tau(r_n)| =|\tau(e)-\tau(r_ne)| = |\tau(e-r_ne)|\leq \Vert e-r_ne\Vert<M'/n$ for some positive constant $M'$ by Lemma~\ref{l:normalized}(3) and Lemma~\ref{l:passman.limits}(2).  We conclude that $|\tau(e)-\Vert r_n\Vert^2|\to 0$ and so $\Vert r_n\Vert^2\to \tau(e)$.  Thus $\tau(e)$ is a nonnegative real number.

For (2), observe that $\Vert e\Vert\leq \Vert e-r_ne\Vert +\Vert r_ne\Vert\leq \Vert e-r_ne\Vert+\Vert r_n\Vert \cdot |e|$ by Lemma~\ref{p:Passman.inequality}.  By Lemma~\ref{l:passman.limits}(2), $\Vert e-r_ne\Vert\to 0$ and by (1), $\Vert r_n\Vert \to \sqrt{\tau(e)}$.  It follows that $\tau(e)\geq\Vert e\Vert ^2/|e|^2>0$.
\end{proof}

This completes the proof of Theorem~\ref{t:stably.finite}.
\end{proof}

\section{Constructing faithful traces on ample groupoid algebras}
Let $\mathscr G$ be an ample groupoid.  Note that $\Gamma_c(\mathscr G\skel 0)$ is the generalized Boolean algebra of compact open subsets of $\mathscr G\skel 0$.  A \emph{mean} on $\mathscr G\skel 0$ is a finitely additive map $\mu\colon \Gamma_c(\mathscr G\skel 0)\to [0,\infty)$.  That is, $\mu(U\cup V) =\mu(U)+\mu(V)$ whenever $U\cap V=\emptyset$.  If $\mathscr G\skel 0$ is compact, then the mean $\mu$ is \emph{normalized} if $\mu(\mathscr G\skel 0)=1$.  The mean $\mu$ is \emph{invariant} if $\mu(U\inv U) = \mu (UU\inv)$ for all $U\in \Gamma_c(\mathscr G)$.  Invariant means have been studied by many authors, not always under the same name, and often in the language of Boolean inverse semigroups~\cite{lawsonetalmean,starlingmean}. Munn and Crabb considered these implicitly quite early on in their study of traces on $E$-unitary inverse semigroups~\cite{CrabbMunn}.

An \emph{invariant measure} on $\mathscr G\skel 0$ is a Radon measure $\nu$ on $\mathscr G\skel 0$ such that $\nu(U\inv U)=\nu (UU\inv)$ for all $U\in \Gamma(\mathscr G)$; the measure $\nu$ is not required to be finite.  If $\mathscr G$ is a Hausdorff \'etale groupoid with compact unit space, then invariant measures give rise to traces on the reduced $C^*$-algebra of $\mathscr G$. (This is true for noncompact unit spaces, as well, provided the invariant measure happens to be finite.)  We will follow the same scheme here.  Sticking to the world of complex algebras avoids integrability issues and will allow us to handle to some degree non-Hausdorff groupoids.

We begin by showing that each invariant mean on $\mathscr G\skel 0$ extends uniquely to an invariant measure.

\begin{Lemma}\label{l:extension}
Let $\mathscr G$ be an ample groupoid and $\mu$ an invariant mean on $\mathscr G\skel 0$.  Then there is a unique invariant measure extending $\mathscr G\skel 0$.  Conversely, if $\mu$ is an invariant measure on $\mathscr G\skel 0$, then $\mu|_{\Gamma_c(\mathscr G\skel 0)}$ is an invariant mean.
\end{Lemma}
\begin{proof}
Since Radon measures are finite on compact sets, it is clear that the restriction of an invariant measure to $\Gamma_c(\mathscr G\skel 0)$ is an invariant mean.  We turn to the converse.

It is standard that if $X$ is a locally compact, Hausdorff and totally disconnected space, then any finitely additive, nonnegative real-valued mapping defind on its Boolean algebra of compact open sets extends uniquely to a Radon measure.  I was unable to find a reference treating the noncompact case, so I am including the proof for completeness.

Let $A$ be the complex $\ast$-algebra of locally constant mappings $f\colon \mathscr G\skel 0\to \mathbb C$ with compact support under pointwise multiplication.   Each element of $A$ can be expressed in the form $f=\sum_{U\in \Gamma_c(\mathscr G\skel 0)}c_U1_U$ (one may even write $f$ as a linear combination of indicator functions of disjoint compact open sets), and so we can define a linear functional $\lambda\colon A\to \mathbb C$ by \[\lambda(f) = \sum_{U\in \Gamma_c(\mathscr G\skel 0)}c_U\mu(U)= \int fd\mu.\]  Here we are exploiting that $f$ is a simple function, and hence the fact  that the integral is well defined and behaves as expected requires only finite additivity and nonnegativity of $\mu$.  Note that if $f\geq 0$, then $\lambda(f)\geq 0$, and so $\lambda (\ov ff)\geq 0$ for any $f\in A$.   Also writing $f=\sum_{i=1}^n c_i1_{U_i}$ with $U_1,\ldots, U_n\in \Gamma_c(\mathscr G\skel 0)$ pairwise disjoint, we see that $|\lambda(f)|\leq \sum_{i=1}^n|c_i|\mu(U_i)=\lambda(|f|)$.  Of course, this integral is only defined for the moment on locally constant functions with compact support.

The approach is then to show that $\lambda$ extends to a positive functional on $C_c(\mathscr G\skel 0)$ (the continuous functions with compact support on $\mathscr G\skel 0$) and apply the Riesz representation theorem.  Let $f\in C_c(\mathscr G\skel 0)$ with $K=\ov{\supp(f)}$, which is a compact set.  Then we can cover $K$ by compact open sets, pass to a finite subcover and take the union of this subcover to find a compact open set $K_0$ with $K\subseteq K_0$.  Since $K_0$ is clopen and compact, we have that $C(K_0)$ is a sub-$\ast$-algebra of $C_c(\mathscr G\skel 0)$ (via extension by $0$ outside of $K_0$) and $f\in C(K_0)$.   It follows that $C_c(\mathscr G\skel 0)$ is the directed union of the subalgebras $C(K_0)$ with $K_0\subseteq \mathscr G\skel 0$ compact open.
Let $A_{K_0}$ be the subalgebra of $A$ consisting of those functions with support contained in $K_0$; this is exactly the algebra of locally constant functions on $K_0$. We first show that there is an extension $\lambda_{K_0}$ of $\lambda|_{A_{K_0}}$ to $C(K_0)$ that is a positive linear functional.  Afterwards, we show that if $K_0\subseteq K_1$, then $(\lambda_{K_1})|_{C(K_0)} = \lambda_{K_0}$, and hence we can extend $\lambda$ to $C_c(\mathscr G\skel 0)$ (via the universal property of a direct limit).

 The $\ast$-algebra $A_{K_0}$ separates points of $K_0$ since $\mathscr G\skel 0$ is Hausdorff with a basis of compact open sets and $K_0$ is compact open in $\mathscr G\skel 0$.  Thus, $A_{K_0}$ is dense in the sup-norm on $C(K_0)$ by the Stone-Weierstrass theorem.  Let $f\in C(K_0)$ and choose a sequence $(f_n)$ of elements of $A_{K_0}$ converging to $f$ in the sup-norm.  We first show that $\lambda(f_n)$ converges; we then show the limit is independent of the chosen sequence $(f_n)$.
  We check that $(\lambda(f_n))$ is a  Cauchy sequence.  Let $\varepsilon>0$ be given.   Choose $N>0$ so that $\vert f_n-f_m\vert_{\infty}<\varepsilon$ for $m,n\geq N$.  Then we compute, for $m,n\geq N$,
 \[\vert \lambda(f_n)-\lambda(f_m)\vert = \left\vert\int_{K_0} f_n-f_m d\mu\right\vert\leq \int_{K_0}\vert f_n-f_m\vert d\mu\leq \varepsilon\mu(K_0).\]  Since $\mu(K_0)\geq 0$ and $\varepsilon$ was arbitrary, we deduce that $(\lambda(f_n))$ is a Cauchy sequence and hence has a limit.  To see that this limit is independent of the choice of $(f_n)$, suppose that $g_n\to f$ with $g_n\in A_{K_0}$ for all $n\geq 1$.  Let $\varepsilon>0$.  We can choose $N>0$ so that $|f_n-f|_{\infty}<\varepsilon/2$ and $|g_n-f|_{\infty}<\varepsilon/2$ for all $n\geq N$.  Then $|f_n-g_n|_{\infty}\leq |f_n-f|+|f-g_n|_{\infty}<\varepsilon$.  Thus we have that
 \[|\lambda(f_n)-\lambda(g_n)| = \left\vert\int_{K_0}f_n-g_n d\mu\right\vert\leq \int_{K_0}\vert f_n-g_n\vert d\mu\leq \varepsilon \mu(K_0).\]  Since $\mu(K_0)\geq 0$ and $\varepsilon$ was arbitrary, we deduce that $\lim \lambda(f_n)=\lim \lambda(g_n)$, yielding a well-defined value $\lambda_{K_0}(f)$.  Clearly, $\lambda_{K_0}$ is a positive functional since if $f_n\to f$, $g_n\to g$ and $c\in \mathbb C$, then $cf_n+g_n\to cf+g$, $\ov f_n\to \ov f$ and $\ov f_nf_n\to \ov ff$.  Finally, note that if $K_0\subseteq K_1$ and $f_n\to f$ in $C(K_0)$, then since we can view $C(K_0)\subseteq C(K_1)$ via extension by zero, we have that $f_n\to f$ in $C(K_1)$.  It then follows that $\lambda_{K_0}(f) =\lim \lambda(f_n) = \lambda_{K_1}(f)$.  We may therefore define $\lambda'\colon C_c(\mathscr G\skel 0)\to \mathbb C$ by $\lambda'(f)=\lambda_{K_0}(f)$ where $K_0$ is any compact open set containing $\supp(f)$ (and there is at least one such, as observed earlier).  If $f\in A$, then $f=\lim f_n$ where $f_n=f$ for all $n\geq 1$ and thus $\lambda'(f)=\lambda(f)$.  Therefore, $\lambda'$ is a positive linear functional extending $\lambda$.

 By the Riesz representation theorem, there is a unique Radon measure $\nu$ on $\mathscr G\skel 0$ with $\lambda'(f) = \int fd\nu$.  Note that if $U\in \Gamma_c(\mathscr G\skel 0)$, then $\nu(U) = \int 1_Ud\nu =\lambda'(1_U)=\lambda(1_U)=\mu(U)$, and so $\nu$ extends $\mu$.  Also observe that if $\mu'$ is any Radon measure extending $\mu$, then $\int fd\mu' = \lambda(f)$ for any  $f\in A$.  But we have already seen that, for any $f\in C_c(\mathscr G\skel 0)$, we can find a compact open set $K_0$ containing $\overline{\supp(f)}$ and a sequence $(f_n)$ in $A_{K_0}$ converging to $f$ in the sup-norm, and so we deduce that $\int fd\mu'=\lim\int f_nd\mu'=\lim \lambda(f_n)= \lambda'(f)$ (by bounded convergence), and hence $\mu'=\nu$ by uniqueness in the Riesz representation theorem.  It remains to show that $\nu$ is invariant.

 Let $U\in \Gamma(\mathscr G)$ and $X=\{V\in \Gamma_c(\mathscr G)\mid V\subseteq U\}$.  Since  $\mathscr G$ has a basis of compact open bisections, we have that $U = \bigcup_{V\in X} V$. Also, if $V,V'\in X$, then $V\cup V'\in X$ since any open subset of a bisection is a bisection. Note that  $U\inv U = \bigcup_{V\in X} V\inv V$ and $UU\inv = \bigcup_{V\in X}VV\inv$.  Each $V\inv V$ with $V\in X$ is compact open. Also, if $K\subseteq U\inv U$ is compact, then $K\subseteq V\inv V$ for some $V\in X$ since the set $X$ is directed upwards and $K$ is compact.  By inner regularity, $\nu(U\inv U)=\sup \nu(K)$ where $K$ runs over the compact subsets of $U\inv U$.  But by the preceding discussion, this sup is, in fact, $\sup_{V\in X}\mu(V\inv V)$.  Similarly, $\nu(UU\inv) = \sup_{V\in X}\mu(VV\inv)$.  Since $\mu$ is invariant, $\mu(V\inv V)=\mu(VV\inv)$ for all $V\in X$, and hence $\nu(U\inv U)=\nu(UU\inv)$.  This completes the proof.
\end{proof}

The above proof shows that any mean on $\mathscr G\skel 0$ can be extended to a unique Radon measure, but we shall only be interested in invariant means.

From now on, we shall abuse notation and write $\mu$ for both the mean and the unique Radon measure extending it.
We note that the invariant  measure extending an invariant mean need not be a finite measure when $\mathscr G\skel 0$ is not compact.

\begin{Prop}\label{p:restriction}
Let $\mathscr G$ be an ample groupoid and $f\in \mathbb C\mathscr G$.
\begin{enumerate}
\item If $\mathscr G$ is Hausdorff, then $f|_{\mathscr G\skel 0}$ is locally constant with compact support.
\item In general, $f|_{\mathscr G\skel 0}$ is Borel measurable, simple and integrable with respect to any Radon measure $\mu$ on $\mathscr G\skel 0$.
\end{enumerate}
\end{Prop}
\begin{proof}
If $\mathscr G$ is Hausdorff, then $\mathscr G\skel 0$ is clopen and $f$ is locally constant. It follows that $f|_{\mathscr G\skel 0}$ is locally constant and $\supp(f|_{\mathscr G\skel 0})=\supp(f)\cap \mathscr G\skel 0$ is compact.  For the second item, we can write $f=\sum_{U\in \Gamma_c(\mathscr G)}c_U1_U$.  Then $f|_{\mathscr G\skel 0} = \sum_{U\in \Gamma_c(\mathscr G)}c_U1_{U\cap \mathscr G\skel 0}$.  Since $\mathscr G\skel 0$ is open, $U\cap \mathscr G\skel 0$ is open and hence Borel.  Moreover, $U\cap \mathscr G\skel 0\subseteq U\inv U$ and $U\inv U\subseteq \mathscr G\skel 0$ is compact open.  Thus $\mu(U\cap \mathscr G\skel 0)\leq \mu(U\inv U)<\infty$ for any Radon measure $\mu$.  It follows that $1_{U\cap \mathscr G\skel 0}$ is Borel measurable and $\mu$-integrable for any $U\in \Gamma_c(\mathscr G)$, whence $f|_{\mathscr G\skel 0}$ is as well. Moreover, $f|_{\mathscr G\skel 0}$ is a simple function being a linear combination of indicator functions of open sets.
\end{proof}

We shall  show how to obtain traces from invariant means or, equivalently, invariant measures.  We follow the same idea commonly used for Hausdorff \'etale groupoids with compact unit spaces in the $C^*$-algebra setting.  In what follows, it will be convenient to describe $(f\ast f^\ast)|_{\mathscr G\skel 0}$ for $f\in \mathbb C\mathscr G$.

\begin{Lemma}\label{l:positive}
Let $f\in \mathbb C\mathscr G$ and $x\in \mathscr G\skel 0$.  Then we have
\begin{equation}\label{eq:formula}
(f\ast f^\ast)(x) =\sum_{\ran(\gamma)=x}|f(\gamma)|^2
\end{equation}
 and hence $\supp((f\ast f^\ast)|_{\mathscr G\skel 0}) = \ran(\supp(f))$.   Moreover, if $\mathscr G$ is Hausdorff, then $\ran(\supp(f))$ is compact open, while in general it is a Borel set, which has finite measure with respect to any Radon measure on $\mathscr G\skel 0$.
\end{Lemma}
\begin{proof}
It follows from the definition that \[(f\ast f^\ast)(x)=\sum_{\ran(\gamma)=x}f(\gamma)\ov {f(\gamma)} = \sum_{\ran(\gamma)=x}|f(\gamma)|^2\] and hence the first claim follows.  The second claim is obvious if $f=0$, and so we suppose $f\neq 0$.  If $\mathscr G$ is Hausdorff, then $\supp(f)$ is compact open and hence $\ran(\supp(f))$ is also compact open.  For the general case, note that if $g=(f\ast f^\ast)|_{\mathscr G\skel 0}$, then $g$ is Borel measurable by Proposition~\ref{p:restriction}, and hence $\ran(\supp(f))=g\inv(\mathbb C\setminus \{0\})$ is Borel.  Moreover, $g\geq 0$ by \eqref{eq:formula}.  Since $g$ is $\mu$-integrable and simple (again by Proposition~\ref{p:restriction}), it takes on only finitely values, and so there is a least nonzero value $m$ taken on by $g$ and $m>0$.  Therefore, $\infty>\int gd\mu\geq \int_{\ran(\supp(f))}md\mu =m\mu(\ran(\supp(f)))$, and so $\ran(\supp(f))$ has finite measure with respect to $\mu$.
\end{proof}

We now construct a trace on $\mathbb C\mathscr G$ from an invariant measure.  In the case that $\mathscr G$ is Hausdorff and $\mathscr G\skel 0$ is compact, this trace extends to the reduced $C^*$-algebra of $\mathscr G$, cf.~\cite{starlingmean}.  If $\mathscr G$ is Hausdorff and the invariant measure is finite, then the trace can still be extended to the reduced $C^*$-algebra of $\mathscr G$.

\begin{Thm}\label{t:traces}
Let $\mathscr G$ be an ample groupoid and $\mu$ an invariant measure on $\mathscr G\skel 0$.  Then $\tau\colon\mathbb C\mathscr G\to \mathbb C$ given by
\[\tau(f) = \int f|_{\mathscr G\skel 0}d\mu\] is a $\Gamma_c(\mathscr G)$-contractive trace.  Moreover, if $\mathscr G\skel 0$ is compact and $\mu$ is normalized, then $\tau$ is normalized.
\end{Thm}
\begin{proof}
It follows from Proposition~\ref{p:restriction} that $\tau(f)$ is a well-defined functional.  Moreover, from basic properties of integration $\tau(f^*) = \int \ov f|_{\mathscr G\skel 0}d\mu = \ov{\tau(f)}$.  It follows from Lemma~\ref{l:positive} that $\tau(f\ast f^\ast)\geq 0$ since $(f\ast f^\ast)|_{\mathscr G\skel 0}\geq 0$.  It remains to prove that $\tau(f\ast g)=\tau(g\ast f)$ for $f,g\in \mathbb C\mathscr G$.

 Since $\mathbb C\mathscr G$ is spanned by the indicator functions $1_U$ with $U\in \Gamma_c(\mathscr G)$ and $1_U\ast 1_V= 1_{UV}$, $1_V\ast 1_U = 1_{VU}$ for $U,V\in \Gamma_c(\mathscr G)$, it suffices to show that $\mu(UV\cap \mathscr G\skel 0) = \mu(VU\cap \mathscr G\skel 0)$ for all $U,V\in \Gamma_c(\mathscr G)$.  Notice that $x\in UV\cap \mathscr G\skel 0$ if and only if there exists $\gamma\in V$ with $\dom(\gamma)=x$ and $\gamma\inv \in U$.  That is $UV\cap \mathscr G\skel 0 = \dom(U\inv\cap V)$.  Similarly, $VU\cap \mathscr G\skel 0= \dom(V\inv \cap U) = \ran(U\inv \cap V)$.  But notice that $U\inv \cap V\in \Gamma(\mathscr G)$, being open and contained in a compact open bisection.  Therefore, by invariance of $\mu$, we have that
\begin{align*}
\mu(\dom(U\inv \cap V))&=\mu((U\inv\cap V)\inv (U\inv \cap V))\\ &=\mu((U\inv\cap V)(U\inv \cap V)\inv)= \mu(\ran(U\inv \cap V)).
\end{align*}
 We may now conclude that $\tau$ is a trace.

 We next verify that $\tau$ is $\Gamma_c$-contractive.  Observe that \[\tau((f\ast1_U)\ast (f\ast 1_U)^\ast) = \tau(f\ast 1_{UU\inv}\ast f^\ast)=\tau(1_{UU\inv}\ast f^\ast\ast f).\] Putting $W=UU\inv$, we have that $1_W\ast g = g|_{\ran\inv(W)}$ for any $g\in \mathbb C\mathscr G$.  Therefore, $(1_W\ast f^\ast\ast f)|_{\mathscr G\skel 0} = (f^\ast\ast f)|_W$.  We conclude (using that $(f^\ast\ast f)|_{\mathscr G\skel 0}\geq 0$ by Lemma~\ref{l:positive}) that
 \begin{align*}
 \tau(f\ast1_U\ast(f\ast 1_U)^\ast)&=\int_{W} (f^\ast \ast f)|_{\mathscr G\skel 0}d\mu \leq \int(f^\ast \ast f)|_{\mathscr G\skel 0}d\mu =\tau(f^\ast \ast f)\\ &=\tau(f\ast f^\ast),
 \end{align*}
 whence $\tau$ is $\Gamma_c(\mathscr G)$-contractive.  The final statement is immediate from the definitions.
\end{proof}

The fact that we are using functions from $\mathbb C\mathscr G$ allows us to avoid any integrability issues when dealing with $\mathscr G\skel 0$ noncompact and with $\mathscr G$ non-Hausdorff.

Following the terminology of~\cite{lawsonetalmean,starlingmean} an invariant mean on $\mathscr G\skel 0$ is called \emph{faithful} if $\mu(U)=0$ implies $U=\emptyset$ for all $U\in \Gamma_c(\mathscr G\skel 0)$.   Equivalently, the mean $\mu$ is faithful if and only if the extension of $\mu$ to a Radon measure is positive on any nonempty open subset.  We shall need a stronger notion to deal with non-Hausdorff groupoids. Let $\mu$ be an invariant mean on $\mathscr G\skel 0$, which we then extend uniquely to an invariant measure.  We say that $\mu$ is \emph{strongly faithful} if $\mu(\ran(\supp(f)))>0$ for each nonzero function $f\in \mathbb C\mathscr G$. 

 We recall here the definition of the singular ideal $J_{\mathbb C}(\mathscr G)$ of $\mathbb C\mathscr G$~\cite{nonhausdorffsimple,simplicity,BGHL25}: \[J_{\mathbb C}(\mathscr G)=\{f\in \mathbb C\mathscr G\mid \supp(f)\ \text{has empty interior}\}.\]   If $\mathscr G$ is Hausdorff, then $J_{\mathbb C}(\mathscr G)=0$ as $\supp(f)$ is open for all $f\in \mathbb C\mathscr G$ because $f$ is locally constant.

\begin{Prop}\label{p:strongly.faithful}
Let $\mathscr G$ be an ample groupoid.  Then any strongly faithful mean on $\mathscr G\skel 0$ is faithful.  If $\mathscr G$ is Hausdorff, or more generally if $J_{\mathbb C}(\mathscr G)=0$, then any faithful mean on $\mathscr G\skel 0$ is strongly faithful.
\end{Prop}
\begin{proof}
If $U\subseteq \mathscr G\skel 0$ is compact open, then $\ran(\supp(1_U))=U$.  Therefore, any strongly faithful mean is faithful.  Suppose now that $J_{\mathbb C}(\mathscr G)=0$ and that $\mu$ is a faithful invariant mean.  If $0\neq f\in \mathbb C\mathscr G$, then $\ran(\supp(f))$ has nonempty interior by~\cite[Lemma~4.1]{BGHL25} (for the Hausdorff case Lemma~\ref{l:positive} applies), and hence $\mu(\ran(\supp(f)))>0$.
\end{proof}

We now establish connections between faithful invariant means, strongly faithful invariant means and faithful traces.

\begin{Thm}\label{t:faithful.trace.gives.mean}
Let $\mathscr G$ be an ample groupoid.  Consider the statements:
\begin{enumerate}
  \item There is a faithful invariant mean on $\mathscr G\skel 0$;
  \item There is a strongly faithful invariant mean on $\mathscr G\skel 0$;
  \item $\mathbb C\mathscr G$ admits a faithful $\Gamma_c(\mathscr G)$-contractive trace;
   \item $\mathbb C\mathscr G$ admits a faithful trace.
\end{enumerate}
Then $(2)\implies (3)\implies (4)\implies (1)$.  If $\mathscr G$ is Hausdorff, or more generally if the singular ideal $J_{\mathbb C}(\mathscr G)$ vanishes, then all four statements are equivalent.
\end{Thm}
\begin{proof}
Trivially, (3) implies (4).
Suppose that $\tau$ is a faithful trace on $\mathbb C\mathscr G$.  We can define a faithful mean $\mu\colon \Gamma_c(\mathscr G\skel 0)\to [0,\infty)$ by $\mu(U) = \tau(1_U)$.  Note that since $1_U$ is a projection, $\mu(U)=\tau(1_U)\geq 0$ with equality if and only if $U=\emptyset$, as $\tau$ is faithful.  If $U,V\in \Gamma_c(\mathscr G\skel 0)$ are disjoint, then $1_{U\cup V} = 1_U+1_V$, and so $\mu(U\cup V)=\tau(1_{U\cup V}) = \tau(1_U)+\tau(1_V) = \mu(U)+\mu(V)$.  Thus $\tau$ is a mean.  Finally, if $U\in \Gamma_c(\mathscr G)$, then $\mu(U\inv U) = \tau(1_{U\inv U})=\tau(1_{U\inv}\ast 1_U)=\tau(1_U\ast 1_{U\inv})=\tau(1_{UU\inv})=\mu(UU\inv)$, and so $\mu$ is invariant.  This shows that (4) implies (1).  If $\mathscr G$ is Hausdorff, then (1) implies (2) by Proposition~\ref{p:strongly.faithful}.

Suppose now that $\mu$ is a strongly faithful mean on $\mathscr G\skel 0$ and extend $\mu$, as usual, to an invariant measure on $\mathscr G\skel 0$.  Then we can define a $\Gamma_c(\mathscr G)$-contractive trace $\tau$ on $\mathbb C\mathscr G$ as per Theorem~\ref{t:traces}.  We claim that $\tau$ is faithful.  Indeed, suppose $0\neq f\in \mathbb C\mathscr G$.  Then we  have that $(f\ast f^\ast)|_{\mathscr G\skel 0}> 0$ with support $\ran(\supp(f))\neq \emptyset$, which is Borel and has finite $\mu$-measure by Lemma~\ref{l:positive}.  Moreover, $(f\ast f^\ast)|_{\mathscr G\skel 0}$  is a simple function by Proposition~\ref{p:restriction}, and hence takes on only finitely many values. Let $m>0$ be the minimum nonzero value that $(f\ast f^\ast)|_{\mathscr G\skel 0}$ takes on.  Then
\begin{align*}
\tau(f\ast f^\ast)&=\int (f\ast f^\ast)|_{\mathscr G\skel 0}d\mu = \int_{\ran(\supp(f))}(f\ast f^\ast)|_{\mathscr G\skel 0}d\mu\\ &\geq \int_{\ran(\supp(f))} md\mu=m\mu(\ran(\supp(f)))>0
\end{align*}
 since $\mu$ is strongly faithful.  Thus $\tau$ is faithful.
\end{proof}

We remark that if $\mathscr G\skel 0$ is compact and $\mathscr G$ is Hausdorff, then (1)--(4) are also equivalent to the existence of a faithful trace on the reduced $C^*$-algebra of $\mathscr G$ by~\cite{starlingmean}.  If we just assume that $\mathscr G$ is Hausdorff, then there is still a faithful trace on the reduced $C^\ast$-algebra of $\mathscr G$ if there is a faithful invariant mean on $\mathscr G\skel 0$ whose extension to an invariant measure is finite.  This is because there is a faithful conditional expectation $P\colon C^*_r(\mathscr G)\to C_0(\mathscr G\skel 0)$ (which on $C_c(\mathscr G)$, and hence $\mathbb C\mathscr G$, is just restriction) and therefore $\tau(f) = \int P(f)d\mu$ gives a faithful trace (faithful since $\mu$ is positive and $P$ is faithful, and a trace because $\mathbb C\mathscr G$ is dense in $C^*_r(\mathscr G)$, cf.~\cite{mygroupoidalgebra}).  I do not know if a strongly faithful invariant mean that extends to a finite measure induces a faithful trace on the reduced $C^*$-algebra in the non-Hausdorff setting.

We now arrive at a generalization of Kaplansky's theorem for group algebras in characteristic $0$~\cite{Kaplanskydirect}.

\begin{Cor}\label{c:stable.finite.strongly.faithful}
Let $\mathscr G$ be an ample groupoid and suppose that there is a strongly faithful invariant mean on $\mathscr G\skel 0$.  Then $K\mathscr G$ is stably finite for all integral domains $K$ of characteristic zero.
\end{Cor}
\begin{proof}
By Corollary~\ref{c:field.is.c}, it suffices to handle the case $K=\mathbb C$. But $\mathbb C\mathscr G$ has a $\Gamma_c(\mathscr G)$-contractive faithful trace by Theorem~\ref{t:faithful.trace.gives.mean}, and so $\mathbb C\mathscr G$ is stably finite by Theorem~\ref{t:stably.finite}.
\end{proof}

The next theorem is inspired by the construction in~\cite{EasdownMunn}.

\begin{Thm}\label{t:finite.orbits.countableD}
Let $\mathscr G$ be an ample groupoid and suppose that there is a countable set $\mathcal O_1,\mathcal O_2,\ldots$ of finite orbits such that $\bigcup_{i\geq 1}\mathcal O_i$ is $\mathbb C$-dense.  Then there is a strongly faithful invariant mean on $\mathscr G\skel 0$ whose extension to an invariant measure is finite.  Hence $K\mathscr G$ is stably finite for any integral domain $K$ of characteristic $0$.
\end{Thm}
\begin{proof}
Let \[\mu =\sum_{i\geq 1}\frac{1}{2^i}\sum_{x\in \mathcal O_i}\frac{1}{|\mathcal O_i|}\delta_x\] where $\delta_x$ is the Dirac measure supported at $x$.  Then $\mu$ is a finite Radon measure on $\mathscr G\skel 0$.  It is invariant, because if $U\in \Gamma(\mathscr G)$, then \[|U\inv U\cap \mathcal O_i| = |\{\gamma\in U\mid \dom(\gamma)\in \mathcal O_i\}|   =|\{\gamma\in U\mid \ran(\gamma)\in \mathcal O_i\}| = |UU\inv\cap \mathcal O_i|.\]
But for any Borel subset of $\mathscr G\skel 0$, we have that $\mu(W) = \sum_{i\geq 1}\frac{|W\cap \mathcal O_i|}{2^i|O_i|}$, and so $\mu(UU\inv)=\mu(U\inv U)$.

To see that $\mu$ is strongly faithful, note that by definition of $\mathbb C$-density, if $0\neq f\in \mathbb C\mathscr G$, then $\ran(\supp(f))\cap \mathcal O_i\neq \emptyset$ for some $i\geq 1$, and so $\mu(\ran(\supp(f)))>0$.

The final statement follows from Corollary~\ref{c:stable.finite.strongly.faithful}.
\end{proof}

The above theorem, of course, recovers Kaplansky's theorem that group algebras are stably finite over fields of characteristic $0$.

Easdown and Munn~\cite{EasdownMunn} constructed a faithful trace on $\mathbb CS$ if $S$ is an inverse semigroup with countably many $\mathscr D$-classes and each $\mathscr D$-class contains finitely many idempotents.  Such inverse semigroups do not need to have Hausdorff groupoids.  So one cannot immediately deduce from their result that there is a strongly faithful invariant mean on $\wh{E(S)}$.  But the mean implicit in their proof does have this property and inspired Theorem~\ref{t:finite.orbits.countableD}, which of course implies their result and recovers the same trace.

\begin{Cor}\label{c:munn.easdown}
Let $S$ be an inverse semigroup with countably many $\mathscr D$-classes such that each $\mathscr D$-class contains finitely many idempotents.  Then $\mathbb CS$ admits a faithful trace.  In fact, there is a strongly faithful invariant mean on $\widehat{E(S)}$, and so $KS$ is stably finite for any integral domain of characteristic $0$.
\end{Cor}
\begin{proof}
The orbits of principal characters form a countable set of $\mathbb C$-dense finite orbits and so Theorem~\ref{t:finite.orbits.countableD} applies.
\end{proof}

Note that in characteristic $0$, Corollary~\ref{c:Munn.result} can be deduced from Corollary~\ref{c:munn.easdown},  which does not directly appeal to Kaplansky's theorem, since $S$ is a direct limit of its finitely generated inverse subsemigroups and these will have a countable set  of $\mathscr D$-classes each of which contains finitely many idempotents.  Note that if $S$ in Corollary~\ref{c:munn.easdown} has a Hausdorff universal groupoid, then the reduced $C^\ast$-algebra of $S$ will have a faithful trace and hence is stably finite.  However, stable finiteness of $\mathbb CS$ for inverse semigroups with finitely many idempotents in each $\mathscr D$-class  cannot be reduced to the case of inverse semigroups with Hausdorff universal groupoids.  P.~V.~Silva and the author~\cite{silvasteinbergFP} have constructed  finitely presented inverse semigroups meeting the hypothesis of Corollary~\ref{c:Munn.result} (in fact, having finite $\mathscr R$-classes), but not having  Hausdorff universal groupoids.  Let us remark that if $S$ is finitely presented and does not have a Hausdorff universal groupoid, then it cannot be expressed as a direct limit of inverse semigroups with Hausdorff universal groupoids.  Indeed, suppose that $S\cong \varinjlim_{\alpha\in D}S_{\alpha}$ with $S$ finitely presented and the $S_{\alpha}$ having Hausdorff universal groupoids.  Let $X\subseteq S$ be a finite generating set such that $S$ is defined by finitely many relations $r_i=r'_i$, for $i=1,\ldots, n$, over the alphabet  $X$ (actually, any finite generating set will do by a standard argument).   For $\alpha\leq \beta$, let $\rho_{\alpha,\beta}\colon S_{\alpha}\to S_{\beta}$ be the map from the directed system and let $\pi_{\alpha}\colon S_{\alpha}\to S$ be the maps coming from $S\cong \varinjlim S_{\alpha}$.  Note that $S=\bigcup_{\alpha\in D}\pi_{\alpha}(S_{\alpha})$, and this union is directed.    Since $X$ is finite, we can find some $\alpha$ such that $X\subseteq \pi_{\alpha}(S_{\alpha})$.  Thus $\pi_{\beta}$ is surjective for all $\beta\geq \alpha$.  Fix for each $x\in X$ a pre-image $x'\in S_{\alpha}$ and let $X'=\{x'\mid x\in X\}$.   Let $FIS(X)$ be the free inverse semigroup on $X$ and let $\gamma\colon FIS(X)\to S_{\alpha}$ be the homomorphism induced by sending $x$ to $x'$ for $x\in X$.    Since $r_1,\ldots, r_n$ together with $r'_1,\ldots, r'_n$ form a finite set of elements, we can find $\alpha_0\geq \alpha$ such that $\rho_{\alpha,\alpha_0}(\gamma(r_i)) = \rho_{\alpha,\alpha_0}(\gamma(r_i'))$ for $i=1,\ldots, n$.  It then follows that $\pi_{\alpha_0}\colon S_{\alpha_0}\to S$ splits via a homomorphism $\psi\colon S\to S_{\alpha_0}$ with $\psi(x)=\rho_{\alpha,\alpha_0}(\gamma(x))$ for $x\in X$ since the elements $\rho_{\alpha,\alpha_0}(\gamma(x))$ with $x\in X$ satisfy the defining relations of $S$.  Suppose now that $s\in S$ and let  $s'=\psi(s)$.  Since $S_{\alpha_0}$ has a Hausdorff universal groupoid, we can find idempotents $e_1,\cdots, e_m\leq s'$ in $S_{\alpha_0}$ such that if $f\leq s'$ with $f$ an idempotent, then $f\leq e_i$ for some $i=1,\ldots, m$ by~\cite[Theorem~5.17]{mygroupoidalgebra}.  Then $f_i=\pi_{\alpha_0}(e_i)\leq s$, for $i=1,\ldots, n$ are idempotents.  Suppose that $f\in E(S)$ with $f\leq s$.  Then $fs=f$ and so $\psi(f)s'=\psi(f)\psi(s)=\psi(fs)=\psi(f)$, whence $\psi(f)\leq s'$.  Therefore, $\psi(f)\leq e_i$ for some $i$, and so $f\leq f_i$.  We deduce that the universal groupoid of $S$ is Hausdorff by~\cite[Theorem~5.17]{mygroupoidalgebra}.

Crabb and Munn~\cite{CrabbMunn} showed that any free inverse semigroup (or monoid) algebra has a faithful trace (without countability assumptions).
Since free inverse semigroups (monoids) are $E$-unitary, they have Hausdorff groupoids, and so Theorem~\ref{t:faithful.trace.gives.mean} also yields that there is a strongly faithful invariant mean on its spectrum.  Actually, the proof of Crabb and Munn essentially constructs an invariant mean, except that they only define it on a semiring of compact open subsets, rather than the whole generalized Boolean algebra.

We end this section with a result proving that if $\mathscr G$ is an ample groupoid with compact unit space, then stable finiteness of $K\mathscr G$ for some commutative ring $K$ implies that the existence of a normalized invariant mean, and hence the existence of a normalized  trace on the complex algebra of $\mathscr G$ and, in the case that $\mathscr G$ is Hausdorff, a tracial state on the reduced $C^*$-algebra of $\mathscr G$.  To do this, we need some $K$-theory and a result of Goodearl and Handelman~\cite{GoodearlHandelman}.

If $R$ is a unital ring, then $K_0(R)$ is the Grothendieck group for the commutative monoid of isomorphism classes of finitely generated projective $R$-modules under the operation of direct sum.  If we define $K_0(R)^+$ to be the submonoid of $K_0(R)$ generated by the classes $[P]$ of finitely generated projective modules $P$, then one has that $[P]=[Q]$ in $K_0(R)^+$ if and only if $P\oplus R^n\cong Q\oplus R^n$ for some $n\geq 0$, i.e., $P,Q$ are stably isomorphic.

Recall that a partially ordered abelian group $G$ is a group with a partial order $\leq$ that is translation-invariant.  The partial order is determined by its positive cone $G^+=\{g\in G\mid g\geq 0\}$, which is a submonoid satisfying $G\cap (-G)=\{0\}$.  Conversely, any submonoid $G^+$ satisfying this property is the positive cone for the partial order $g\leq h$ if $h-g\in G^+$.  A \emph{strong order unit} in $G$ is an element $u>0$ such that for any $x\in G$, there is a positive integer $n$ with $x\leq nu$.  For example, $1$ is a strong order unit for the additive group of real numbers. One says that $G$ is \emph{directed} if it directed as a poset.  This is well known to be equivalent to $G^+$ generating $G$ as a group. 

\begin{Prop}[Goodearl-Handelman]\label{p:stably.finite.ordered}
Let $R\neq 0$ be a stably finite ring.  Then $K_0(R)$ is a directed abelian group with positive cone $K_0(R)^+$.  Moreover, $[R]$ is a strong order unit in $K_0(R)$.
\end{Prop}

The proof of this in~\cite[Proposition~2.1]{GoodearlHandelman} is straightforward.  No nonzero finitely generated projective module is stably isomorphic to the trivial module because if $P\oplus R^n\cong R^n$ with $n\geq 1$ and $P\neq 0$, then there is a noninvertible surjective homomorphism $R^n\to R^n$ (with kernel $P$), and so $M_n(R)\cong \End_R(R^n)$ is not Dedekind-finite, a contradiction.  In particular, if $[P]=-[Q]$ with $P,Q$ finitely generated projectives, then $[P\oplus Q]=0$ and so $P,Q=0$ by the above. The strong order unit property follows because if $P$ is finitely generated projective, say generated by $n$ elements, then $P\oplus Q=R^n$ for some $n$-generated projective $Q$, and so $n[R]-[P]=[Q]\geq 0$.

If $(G,u)$ is a partially ordered abelian group with strong order unit $u$, then a \emph{state} on $(G,u)$ is an order-preserving additive homomorphism $f\colon G\to \mathbb R$ with $f(u)=1$.  The term ``functional'' is used in~\cite{GoodearlHandelman}; the term ``state'' appears in~\cite{Goodearlreg} and the majority of the literature.  The following crucial result is~\cite[Corollary~3.3]{GoodearlHandelman}.

\begin{Thm}[Goodearl-Handelman]\label{t:state.exists}
Let $G$ be a directed partially ordered abelian group with strong order unit $u$.  Then there exists a state on $(G,u)$.
\end{Thm}

We recall~\cite{mygroupoidalgebra} that $K\mathscr G$ is unital if and only if $\mathscr G\skel 0$ is compact for an ample groupoid $\mathscr G\skel 0$.

\begin{Thm}\label{t:exists.invariant.mean}
Let $\mathscr G$ be an ample groupoid with compact unit space.  Suppose that $K\mathscr G$ is stably finite for some commutative ring $K$.  Then there exists a normalized invariant mean on $\mathscr G\skel 0$.  Consequently, $\mathbb C\mathscr G$ admits a $\Gamma_c(\mathscr G)$-contractive normalized trace, and if $\mathscr G$ is Hausdorff, then  $C^*_r(\mathscr G)$ admits a tracial state.
\end{Thm}
\begin{proof}
Put $R=K\mathscr G$.  Then $K_0(R)$ is a directed partially ordered abelian group with strong order unit $[R]$ by Proposition~\ref{p:stably.finite.ordered}.  Therefore, there is a state $f$ on $(K_0(R),[R])$.  Define $\mu$ by $\mu(U) = f([R1_U])$ for $U\subseteq \mathscr G\skel 0$ compact open.  This makes sense because $1_U$ is idempotent, and so $R1_U$ is a finitely generated projective module.  Also, since $[R1_U]\in K_0(R)^+$, we have that $\mu(U)=f([R1_U])\geq 0$. If $U\cap V=\emptyset$, then $1_{U\cup V}= 1_U+1_V$ and $1_U\ast 1_V=1_V\ast 1_U=1_{U\cap V}=0$, and so $1_U,1_V$ are orthogonal idempotents.  Therefore, $R1_{U\cup V}\cong R1_U\oplus R1_V$, and so $\mu(U\cup V)=f(R1_{U\cup V}) = f([R1_U]+[R1_V])=f([R1_U])+f([R1_V])=\mu(U)+\mu(V)$, establishing that $\mu$ is a mean.
  Moreover, $\mu(\mathscr G\skel 0)=f([R])=1$, and so $\mu$ is normalized.  To check invariance, let $U\in \Gamma_c(\mathscr G)$.  Then $U\inv U$ and $UU\inv$ are compact open subsets of $\mathscr G\skel 0$ and $R1_{UU\inv}\cong R1_{U\inv U}$ as left $R$-modules via $f\mapsto f\ast 1_U$ (with inverse $g\mapsto g\ast 1_{U\inv}$) since $1_{UU\inv} = 1_U\ast 1_{U\inv}$ and $1_{U\inv U}=1_{U\inv}\ast 1_U$.  It follows that $\mu(UU\inv) = f([R1_{UU\inv}])=f([R1_{U\inv U}]) = \mu(U\inv U)$, and so $\mu$ is a normalized invariant mean.  We conclude that $\mathbb C\mathscr G$ admits a normalized trace by   Theorem~\ref{t:traces}.  The final statement follows from the results of Starling~\cite{starlingmean}, which show that a normalized invariant mean on $\mathscr G\skel 0$ yields a tracial state on $C^*_r(\mathscr G)$ when $\mathscr G$ is Hausdorff with compact unit space.
\end{proof}

It would be interesting to compare our work to that of Va\v{s} on traces for Leavitt path algebras over $\ast$-rings~\cite{Vasdirect}.  I believe that for the case of the field of complex numbers, her construction of so-called canonical traces is equivalent to constructing faithful invariant means on the associated gropoid.  I believe that a theory of traces of algebras of Hausdorff ample groupoids over arbitrary $\ast$-rings, generalizing the case of Leavitt path algebras~\cite{Vasdirect}, can be made to work because we only ever need to integrate locally constant functions against an invariant mean, and so likely one can make sense of this over arbitrary $\ast$-rings with an appropriate notion of positivity (as considered in~\cite{Vasdirect}).  For non-Hausdorff groupoids, we need to integrate Borel functions and extend the mean to a measure, and so it less clear to me that one can develop a theory over arbitrary $\ast$-rings.

\section*{Appendix: stable finiteness of algebras of stable semigroups}
This appendix generalizes Corollary~\ref{c:Munn.result} to a class of stable semigroups, extending the main result of~\cite{MunnDirect}.  The approach here does not use groupoids at all, but rather specialized techniques from semigroup theory, and so for this reason we have isolated it as an appendix.  We assume familiarity with Green's relations and preorders~\cite{Green} and structural semigroup theory as found in~\cite{CP,CP2} or~\cite[Appendix~A]{qtheor}.  We also follow the standard notation for equivalence classes of Green's relations.

An element  $s$ of a semigroup $S$ is (von Neumann) \emph{regular} if there exists $t\in S$ with $sts=s$. A semigroup is \emph{regular} if all its elements are regular. Inverse semigroups are precisely the regular semigroups with commuting idempotents.  The multiplicative semigroup of a von Neumann regular ring is, of course, von Neumann regular.

A semigroup $S$ is \emph{stable} if $s\J st$ implies $s\R st$ and $s\J ts$ implies $s\eL ts$.  Finite semigroups, as well as periodic semigroups and compact semigroups, are stable.  If $S$ is a regular semigroup such that each $\mathscr J$-class has either a minimal $\mathscr R$-class or a minimal $\mathscr L$-class, then $S$ is stable by the proofs of~\cite[Theorems~6.45 \& 6.48]{CP2}.  In a stable semigroup, Green's relations $\mathscr J$ and $\mathscr D$ coincide.  A stable semigroup does not contain a bicyclic monoid.  More on stable semigroups can be found in~\cite[Appendix A]{qtheor}.  In particular, if $J$ is a $\mathscr J$-class of a stable semigroup, then either $J^2\cap J=\emptyset$ or each element of $J$ is regular.  A $\mathscr J$-class of the former type is called \emph{null} and of the latter type is called \emph{regular}.  Each $\mathscr L$- and $\mathscr R$-class of a regular $\mathscr J$-class contains an idempotent.

Let $e$ be an idempotent of $S$.  Then  the maximal subgroup $H_e$, which is the $\mathscr H$-class of $e$,  acts freely on the right of $L_e$ and the orbits are exactly the $\mathscr H$-classes contained in $L_e$; these orbits are in bijection with the $\mathscr R$-classes of $D_e$.
Dually, $H_e$ acts freely on the left of $R_e$ and the orbits are described similarly.

With this preparation, we may now prove the main result of this appendix, generalizing~\cite{MunnDirect}.  Our proof is inspired by  Munn's with some improvements.

\begin{Thm}\label{t:munn.gen}
Let $R$ be a unital ring and $S$ a stable semigroup such that each regular $\mathscr J$-class of $S$ contains either finitely many $\mathscr L$-classes or finitely many $\mathscr R$-classes.  Then $RS$ is stably finite if and only if $RG$ is stably finite for each maximal subgroup $G$ of $S$. In particular, if $R$ is an integral domain of characteristic $0$ or if $R$ is locally residually Noetherian (e.g., $R$ is commutative) and each maximal subgroup of $S$ is sofic, then $RS$ is stably finite.
\end{Thm}
\begin{proof}
If $G$ is a maximal subgroup of $S$, then $RG$ is a subring of $RS$ and so it must be stably finite if $RS$ is stably finite.  Suppose now that $RG$ is stably finite for each maximal subgroup $G$ of $S$.    We show that $RS$ is stably finite.  In fact, we first show that $RS$ is Dedekind finite and then explain afterwards how to obtain the general result from this.

 Let $e$ be an idempotent of $RS$ and $a,b\in eRSe$ such that $ab=e$ but $ba\neq e$.  Then $f=e-ba$ is a non-zero idempotent.  Choose a maximal $\mathscr J$-class $J$ intersecting the finite support of $f$ (in the usual ordering on $\mathscr J$-classes).  Let $I=\{s\in S\mid J\nleq_{\mathscr J} J_s\}$.  Then $I$ is an ideal of $S$ and hence spans an ideal $RI$ of $RS$.  Consider $A=RS/RI\cong R_0[S/I]$ where $R_0[S/I]$ is the contracted semigroup algebra of $S/I$ (i.e., the zeroes of $R$ and $R[S/I]$ are identified).  So $A$ has basis $S\setminus I$ (dropping the coset notation) and the multiplication is defined on the basis by the rule $s\cdot t=0$ if $st\in I$, and otherwise $s\cdot t=st$.   Let $\pi\colon RS\to A$ be the projection (which can be viewed as sending elements of $S\setminus I$ to themselves and elements of $I$ to $0$).  Note that $\pi(f)\neq 0$ by construction.  By maximality of $J$, we have that the support of $f$ is contained in $J\cup I$ and hence $\pi(f)\in RJ$. Since $\pi(f)^2=\pi(f)\neq 0$, we cannot have $J^2\subseteq I$ and so $J^2\cap J\neq \emptyset$ (as $J\cup I$ is an ideal of $S$).  Thus $J$ is a regular $\mathscr J$-class.

Without loss of generality (because we may replace $S$ by its opposite semigroup) we may assume that $J$ has finitely many $\mathscr R$-classes.   Notice that $RJ$ is a two-sided ideal of $A$ (since $I\cup J$ is an ideal of $S$).  Moreover, the stability of $S$ implies that $RL_f$ is a left ideal of $A$ for each idempotent $f\in J$.  Indeed, if $s\in S$ and $x\in L_f$, then $sx\in J\cup I$ and, by stability, $sx\in J$ if and only if $sx\in L_f$.  Thus in $A$,  $sx\in RL_f$ (possibly  $0$).  If $f'$ is another idempotent in $J$, then since $\mathscr J=\mathscr D$ on $S$ by stability, Green's lemma implies there exists $s\in S$ such that $x\mapsto xs$ a bijection from $L_f$ to $L_{f'}$.
But then right multiplication by $s$ yields a left $A$-module isomorphism $RL_f\to RL_{f'}$.  Since the $\mathscr L$-classes in $J$ are disjoint, each contains an idempotent (by regularity) and their union is $J$, we conclude that $RJ$ is the direct sum of the $RL_{f'}$ as $f'$ ranges over a set of idempotent representatives of the $\mathscr L$-classes of $J$, and this in turn is isomorphic to a direct sum of copies of $RL_f$ (of cardinality the number of $\mathscr L$-classes in $J$, which might be infinite).    In particular, if $a\in A$, then $aRJ=0$ if and only if $aRL_f=0$.

Suppose that $J$ has $n$ $\mathscr R$-classes. Then $H_f$ acts freely on the right of $L_f$ with $n$ orbits.  Thus $RL_f$ is a free right $RH_f$-module of rank $n$.  Moreover, the action of $A$ on the left of $RL_f$ commutes with the action of $RH_f$ on the right, and hence there is a corresponding homomorphism $A\to \mathrm{End}_{RH_f}(RL_f)$.  Composing with $\pi$ and using that $\mathrm{End}_{RH_F}(RL_f)\cong M_n(RH_f)$ (since $RL_f$ is free of rank $n$ over $RH_f$), we obtain a homomorphism $\rho\colon RS\to M_n(RH_f)$ (which is just the linear extension of the classical left Sch\"utzenberger representation by monomial matrices over $H_f$~\cite{CP}) with the property that $\rho(x)=0$ if and only if $\pi(x)RL_f=0$, if and only if $\pi(x)RJ=0$.  Note that since $\pi(f)\in RJ$ and $\pi(f)\pi(f)=\pi(f)\neq 0$, it follows $\pi(f)$ does not annihilate $RJ$.  Thus $0\neq \rho(f)=\rho(e)-\rho(b)\rho(a)$.  Therefore, $\rho(a)\rho(b)=\rho(e)$, $\rho(b)\rho(a)\neq \rho(e)$ and $\rho(a),\rho(b)\in \rho(e)M_n(RH_f)\rho(e)$.  This contradicts the hypothesis that $RH_f$ is stably finite.  Hence, we must have that $RS$ is Dedekind finite.

To obtain the general result, let $B_n$ be the semigroup of $n\times n$-matrix units, together with $0$; it is a finite inverse semigroup.   Note that $RB_n\cong M_n(R)\times R$ via the map sending $0$ to $(0,1)$ and $E_{ij}$ to $(E_{ij}, 1)$.  It follows that if $S$ is a semigroup, then $R[S\times B_n]\cong RS\otimes_R RB_n\cong RS\otimes_R (M_n(R)\times R)\cong M_n(RS)\times RS$.  Since Green's relations are computed coordinatewise and $B_n$ is finite,  if $S$ is stable and each regular $\mathscr J$-class of $S$ has either finitely many $\mathscr R$- or $\mathscr L$-classes, then the same is true for $S\times B_n$.  Moreover, since $B_n$ has only trivial maximal subgroups, the maximal subgroups of $S\times B_n$ are the same (up to isomorphism) as those of $S$.  Therefore, $R[S\times B_n]$ is Dedekind finite by the above, and so $M_n(RS)$ is Dedekind finite, being isomorphic to a subring.  This concludes the proof that $RS$ is stably finite.
\end{proof}

Munn proved the analogue of this result for Dedekind finiteness under the stronger assumption that either $R$ is a field of characteristic $0$ or $R$ is field and each maximal subgroup is abelian~\cite{MunnDirect}.

The following corollary generalizes Corollary~\ref{c:Munn.result} (where we note that, for an inverse semigroup, each $\mathscr J$-class having finitely many $\mathscr L$- or $\mathscr R$-classes is equivalent to each $\mathscr D$-class having finitely many idempotents).

\begin{Cor}
Let $R$ be a unital ring and $S$ a regular semigroup such that each $\mathscr J$-class of $S$ contains either finitely many $\mathscr L$-classes or finitely many $\mathscr R$-classes.  Then $RS$ is stably finite if and only if $RG$ is stably finite for each maximal subgroup $G$ of $S$.  In particular, if $R$ is an integral domain of characteristic $0$ or if $R$ is locally residually Noetherian (e.g., $R$ is commutative) and each maximal subgroup of $S$ is sofic, then $R$S is stably finite.
\end{Cor}
\begin{proof}
These hypotheses imply that each $\mathscr J$-class of $S$ contains either a minimal $\mathscr L$-class or a minimal $\mathscr R$-class.  Therefore,  $S$ is stable by the proofs of~\cite[Theorems~6.45 \& 6.48]{CP2}.  Thus Theorem~\ref{t:munn.gen} applies.
\end{proof}

\def\malce{\mathbin{\hbox{$\bigcirc$\rlap{\kern-7.75pt\raise0,50pt\hbox{${\tt
  m}$}}}}}\def\cprime{$'$} \def\cprime{$'$} \def\cprime{$'$} \def\cprime{$'$}
  \def\cprime{$'$} \def\cprime{$'$} \def\cprime{$'$} \def\cprime{$'$}
  \def\cprime{$'$} \def\cprime{$'$}


\end{document}